\documentclass[11pt]{article}

\usepackage{graphicx,enumitem}
\usepackage{graphicx,psfrag,epsf}
{%
\centering\addtocounter{figure}{1}
\begin{enumerate}[%
itemsep=2pt,parsep=0em,
label={(\alph*)},
ref={\thefigure. (\alph*)}
]}%
{\end{enumerate}\addtocounter{figure}{-1}}

\usepackage[margin=1in]{geometry}

\usepackage[utf8]{inputenc}
\usepackage{amsmath}
\usepackage[english]{babel}
\usepackage{amssymb}
\usepackage{amsfonts}
\usepackage{bm}
\usepackage{amsthm}
\usepackage{xcolor}
\usepackage[hidelinks=true]{hyperref}
\usepackage{parskip}
\usepackage{multicol}
\usepackage{float}

\newtheorem{cor}{Corollary}[section]
\newtheorem{proposition}{Proposition}[section]

\numberwithin{equation}{section}
\numberwithin{figure}{section}

\usepackage{manyfoot}
\DeclareNewFootnote{A}[alph]
\DeclareNewFootnote{B}[arabic]

\usepackage{natbib}
\hypersetup{urlcolor=blue, citecolor=blue, colorlinks=true, linkcolor=blue}

\begin{document}

\def\spacingset#1{\renewcommand{\baselinestretch}%
{#1}\small\normalsize} \spacingset{1}

\fontsize{11}{14pt plus.8pt minus .6pt}\selectfont \vspace{0.8pc}
\centerline{\Large \textbf{Karhunen-Lo\`eve Expansions for Axially Symmetric Gaussian}}
\vspace{4pt} \centerline{\Large \textbf{Processes: Modeling Strategies and $L^2$ Approximations}}
\vspace{.8cm} 

\begin{center}
Alfredo Alegr\'ia\footnoteB{\baselineskip=10pt  Departamento de Matem\'atica, Universidad T{\'e}cnica Federico Santa Mar{\'i}a, Chile. }$^{,}$\footnoteA{Corresponding author. Email: alfredo.alegria@usm.cl}  
    and    Francisco Cuevas-Pacheco\footnoteB{\baselineskip=10pt  Department of Mathematics, Universit\'e du Qu\'ebec \`a Montreal, Canada.}
\end{center}

\vspace{1cm}

\begin{abstract}
\noindent Axially symmetric processes on spheres, for which the second-order dependency structure may substantially vary with shifts in latitude, are a prominent alternative to model the spatial uncertainty of natural variables located over large portions of the Earth. In this paper, we focus on Karhunen-Lo\`eve expansions of axially symmetric Gaussian processes. First, we investigate a parametric family of Karhunen-Lo\`eve coefficients that allows for versatile spatial covariance functions. The isotropy as well as the longitudinal independence can be obtained as limit cases of our proposal. Second, we introduce a strategy to render any longitudinally reversible process irreversible, which means that its covariance function could admit certain types of asymmetries along longitudes. Then, finitely truncated Karhunen-Lo\`eve expansions are used to approximate axially symmetric processes. For such approximations, bounds for the $L^2$-error are provided. Numerical experiments are conducted to illustrate our findings.\\


\end{abstract}

\section{Introduction}

Stochastic processes on spheres provide a valuable mathematical framework to capture the spatial uncertainty of geophysical processes located over large portions of the Earth \citep{marinucci2011random}. Global data are typically characterized by dissimilar behaviors in different parts of the world, which can be attributed to diverse factors, including wind directions and teleconnections. As a result, the search for sophisticated models for globally dependent data has attracted growing interest from statisticians in recent decades. We refer the reader to \cite{jeong2017spherical} and \cite{porcu2018modeling} for thorough reviews about this topic.

The assumption of isotropy, commonly used in spatial data analysis, implies that the statistical properties of the process do not vary for different points on the surface of a sphere. The literature on isotropic processes is substantial. For instance, the design of parametric families of covariance functions has been addressed by \cite{gneiting2013strictly}, \cite{guinness2016isotropic}, \cite{peron2018admissible} and \cite{alegria2018family}. \cite{lang2015isotropic}, \cite{hansen2015gaussian} and \cite{de2018regularity} discussed the regularity properties of Gaussian processes on spheres and hyperspheres. Computationally efficient simulation algorithms have been proposed by \cite{creasey2018fast}, \cite{cuevas2020fast}, \cite{lantuejoul2019}, \cite{emery2019turning}, \cite{emery2019simulating} and \cite{alegria2020turning}. Although isotropy considerably simplifies the modeling of processes on spheres, it is generally a questionable assumption, as it does not allow for spatially varying dependencies \citep{stein2007spatial}.

In recent years, there has been renewed interest in a more flexible class of processes, referred to as axially symmetric processes \citep{jones1963stochastic}, for which the spatial dependency is stationary with respect to longitude but may substantially change with shifts in latitude. While \cite{hitczenko2012some}, \cite{huang2012simplified} and \cite{bissiri2020} studied several theoretical aspects of axially symmetric processes, \cite{vanlengenberg2019data} and \cite{Emery2019} paid particular attention to the formulation of fast and efficient simulation algorithms. Other authors, including \cite{stein2007spatial}, \cite{jun2008nonstationary}, \cite{castruccio2014beyond}, \cite{castruccio2016assessing}, and \cite{porcu2019axially}, have illustrated the relevance of axially symmetric processes in environmental and climatological applications.

The Karhunen-Lo\`eve expansion in terms of spherical harmonic functions is a convenient mathematical framework to analyze Gaussian processes on spheres \citep{jones1963stochastic,marinucci2011random}. In the axially symmetric scenario, Karhunen-Lo\`eve expansions have received little attention, with the works of \cite{stein2007spatial}, where axially symmetric models are fitted to total column ozone data, and \cite{hitczenko2012some}, who concentrated on models based on differential operators, being notable exceptions. This paper is devoted to the study of axially symmetric Gaussian processes through their Karhunen-Lo\`eve expansions. The main contributions of this work are listed below:
\begin{enumerate}

\item  We investigate a parametric family of Karhunen-Lo\`eve coefficients that allows for flexible second-order dependency structures. Our proposal permits us to gradually go from processes that are constant along the parallels of latitude (longitudinal independence) to processes whose finite-dimensional distributions are invariant under spatial rotations (isotropy).

\item We propose a general and simple strategy to render any longitudinally reversible model irreversible, i.e.; our approach allows covariance functions to be built with certain types of asymmetries along longitudes.

\item We focus on the approximation of axially symmetric processes through finitely truncated Karhunen-Lo\`eve expansions. We provide a theoretical bound for the $L^2$-error associated with this approximation. Such an approximation suggests a natural simulation method, which is examined through numerical experiments.

\end{enumerate}

The article is organized as follows. Section \ref{preliminaries} provides background material on axially symmetric processes and their Karhunen-Lo\`eve expansions. Section \ref{main} contains the main results of this work. Specifically, we propose a parametric family of Karhunen-Lo\`eve coefficients that connects the isotropic case and the case of longitudinal independence. We present a strategy for building longitudinally irreversible processes. The bounds for the $L^2$-error of truncated Karhunen-Lo\`eve expansions are derived as well. In Section \ref{simulacion}, our findings are illustrated through numerical experiments. Section \ref{discussion} concludes the paper with a discussion.

\section{Preliminaries}
\label{preliminaries}

\subsection{Spherical Harmonic Functions}

The aim of this section is to introduce preliminary material about spherical harmonic functions. We denote the latitude and longitude coordinates of a spatial point on $\mathbb{S}^2 = \{\bm{x}\in\mathbb{R}^3: \|\bm{x}\| = 1\}$ by $L \in [0,\pi]$ and $\ell\in [0,2\pi)$, respectively. For two locations on $\mathbb{S}^2$, with coordinates $(L_1,\ell_1)$ and $(L_2,\ell_2)$, the great circle distance between them is given by
$$\text{d}_{\text{GC}}\left(L_1,L_2,\Delta \ell\right) = 2 \arcsin\left\{  \left[ \sin^2\left( \frac{L_1-L_2}{2}\right) + \sin L_1 \sin L_2 \sin^2\left( \frac{\Delta \ell}{2} \right) \right] ^{1/2}\right\},$$
where $\Delta \ell = \ell_1-\ell_2$. This metric represents the length of the shortest arc joining two spherical locations, so it is always true that $\text{d}_{\text{GC}}\left(L_1,L_2,\Delta \ell\right)  \in [0,\pi]$.

Spherical harmonic functions, denoted by $\mathcal{Y}_{nm}(L,\ell)$, for $n\in\mathbb{N}_0$ and $m\in\{-n,\hdots,n\}$,
form an orthogonal basis of the Hilbert space of complex-valued square integrable functions on $\mathbb{S}^2$. When $n\in\mathbb{N}_0$ and $m\in\{0,\hdots,n\}$, we have
$$\mathcal{Y}_{nm} (L,\ell)  =  \sqrt{ \frac{2n+1}{4\pi} \frac{(n-m)!}{(n+m)!} } P_{nm}(\cos L) \exp(\imath m \ell),$$
for $(L,\ell)\in [0,\pi]\times [0,2\pi)$, where $P_{nm}$ is the associated Legendre polynomial \citep{abramowitz1988handbook} and $\imath\in\mathbb{C}$ is the complex unit. For $n\in\mathbb{N}$ and $m\in\{-n,\hdots,-1\}$, spherical harmonic functions are instead given by
$$\mathcal{Y}_{nm} (L,\ell)  =    (-1)^m \overline{\mathcal{Y}_{n-m}} (L,\ell),$$
where $\overline{c}$ denotes the complex conjugate of $c$. The addition theorem is a mathematical identity of great importance when dealing with spherical harmonic functions (see, e.g., \citealp{marinucci2011random}), which states that
\begin{equation*}
P_n\left(\cos\text{d}_{\text{GC}}\left(L_1,L_2,\Delta \ell\right)  \right)   =    \frac{4\pi}{2n+1}   \sum_{m=-n}^n \mathcal{Y}_{nm}(L_1,\ell_1)  \overline{\mathcal{Y}_{nm}}(L_2,\ell_2),   
\end{equation*}
or equivalently,
\begin{equation}
\label{addition_thm} 
P_n\left(\cos \text{d}_{\text{GC}}\left(L_1,L_2,\Delta \ell\right)  \right)   =   P_n(\cos L_1) P_n(\cos L_2)  + 2 \sum_{m=1}^n \frac{(n-m)!}{(n+m)!} \cos(m \Delta \ell) P_{nm}(\cos L_1) P_{nm}(\cos L_2),
\end{equation}
with $P_n = P_{n0}$ standing for the Legendre polynomial of degree $n$ \citep{abramowitz1988handbook}. A more comprehensive discussion on spherical harmonic functions and Fourier analysis on $\mathbb{S}^2$ can be found in \cite{marinucci2011random}.

\subsection{Axially Symmetric Processes}

A zero-mean real-valued Gaussian process, $Z(L,\ell)$, which is indexed by latitude $L\in[0,\pi]$ and longitude $\ell\in[0,2\pi)$, with finite second-order moments, defined on a probability space $(\Omega,\mathcal{F},P)$, is referred to as an axially symmetric process \citep{jones1963stochastic, stein2007spatial} if its covariance function can be written as

\begin{equation}
\label{cov_as}
\text{cov}\{Z(L_1,\ell_1), Z(L_2,\ell_2)\} = C(L_1,L_2,\Delta \ell),
\end{equation}
for some function $C: [0,\pi]^2\times [-2\pi,2\pi] \rightarrow \mathbb{R}$. The covariance function of an axially symmetric process is stationary with respect to longitude and may have heterogeneous behaviours along latitudes. Two important particular cases are discussed.
\begin{description}
\item \textbf{Isotropy.} When the covariance function of the process is a function of locations $(L_1,\ell_1)$ and $(L_2,\ell_2)$ only through their great-circle distance, $\text{d}_{\text{GC}}\left(L_1,L_2,\Delta \ell\right)$, the process is called isotropic. It is clear that (\ref{cov_as}) includes an isotropic structure as a special case. The finite-dimensional distributions of isotropic Gaussian processes are invariant under the group of rotations on $\mathbb{S}^2$ \citep{marinucci2011random}.

\item \textbf{Longitudinal Independence.} Another limit scenario of axial symmetry is the longitudinal independence presented by \cite{Emery2019}, which means that $C(L_1,L_2,\Delta \ell)$ in (\ref{cov_as}) does not depend on $\Delta \ell$. \cite{Emery2019} showed that longitudinally independent processes are constant along the parallels of latitude and argued that they can be useful in structural geology and geotechnics.
\end{description}
The presence or absence of longitudinal symmetry in the covariance function (\ref{cov_as}) provides a classification for axially symmetric processes. Following \cite{stein2007spatial}, an axially symmetric process is called longitudinally reversible if $$C(L_1,L_2,\Delta \ell) = C(L_1,L_2,-\Delta \ell),$$
for every $(L_1,L_2,\Delta \ell)\in [0,\pi]^2 \times [-2\pi,2\pi]$. However, in general, $C$ need not be symmetric in $\Delta \ell$, in which case the process is said to be longitudinally irreversible.

\subsection{Karhunen-Lo\`eve Expansions of Axially Symmetric Processes}

An approach developed by \cite{jones1963stochastic} states that axially symmetric Gaussian processes on $\mathbb{S}^2$ admit Karhunen-Lo\`eve expansions in terms of spherical harmonic functions. Consider the expansion
\begin{equation}
\label{proceso0}
Z(L,\ell) = \sum_{n=0}^{\infty}  \sum_{m=-n }^n   c_{nm} \mathcal{Y}_{nm}(L,\ell),
\end{equation}
where $c_{nm}$ are zero-mean Gaussian random variables. The convergence in (\ref{proceso0}) holds in the $L^2$ sense (the same comment applies for similar series throughout the manuscript). To obtain a real-valued process, the condition $c_{nm} = (-1)^m \overline{c_{n-m}}$ must be imposed, where in particular we have that $c_{n0}$ is a real-valued random variable. Under this symmetry condition and the explicit expressions for spherical harmonic functions, it is evident (\ref{proceso0}) reduces to
\begin{equation}
\label{proceso}
Z(L,\ell)  =  \sum_{n=0}^{\infty}   a_{n0}   \widetilde{P}_{n0}(\cos L)  + 2 \sum_{n=1}^\infty \sum_{m=1}^n   \bigg\{ a_{nm} \cos(m\ell)+ b_{nm} \sin(m\ell) \bigg\}     \widetilde{P}_{nm}(\cos L),
\end{equation}
where $a_{nm}$ and $b_{nm}$ are zero-mean Gaussian random variables, representing real and imaginary parts of $c_{nm}$, respectively, and $$\widetilde{P}_{nm} = \sqrt{ \frac{2n+1}{4\pi} \frac{(n-m)!}{(n+m)!} }   P_{nm}.$$
To ensure that $Z(L,\ell)$ is axially symmetric, the (bi) sequences of coefficients, $\{a_{nm}\}$ and $\{b_{nm}\}$, must be uncorrelated in the index $m$. Specifically, \cite{jones1963stochastic} considered the following conditions:
\begin{description}
\item[{(C1)}] $\text{cov}\{ a_{n0}, a_{n'0} \} = \text{cov}\{ b_{n0}, b_{n'0} \} =   f_{0}(n,n')$, for all $n,n'\geq 0$.
\item[{(C2)}] $\text{cov}\{ a_{nm}, a_{n'm'} \} = \text{cov}\{ b_{nm}, b_{n'm'} \} =  \delta_{m}^{m'} f_{m}(n,n')/2$, for all $n,n'\geq m$, with $m>0$.
\item[{(C3)}] $\text{cov}\{ a_{nm}, b_{n'm'} \} =- \text{cov}\{b_{nm},a_{n'm'}\} =  \delta_{m}^{m'} g_{m}(n,n')/2$, for all $n,n'\geq m$, with $m>0$.
\end{description}
Here, $\delta_{m}^{m'}$ denotes the Kronecker delta, $f_m(n,n')$ captures the covariance function of each individual sequence, and $g_m(n,n')$ characterizes the cross-covariance function between these sequences. Following \cite{jones1963stochastic}, the cross-covariance function satisfies the identity $g_m(n,n') = - g_m(n',n)$, and so $g_m(n,n) = 0$. Defining
$$\bm{\upsilon}_m =(a_{n_1m},  a_{n_2m}, \hdots, b_{n_1m}, b_{n_2m}, \hdots)^\top,$$
for integers $n_1,n_2,\hdots$ being greater than or equal to $m$, we observe that the covariance matrix of $\bm{\upsilon}_m$ can be written as
\begin{equation}
\label{big-matrix}
 \bm{\Gamma}_m= \begin{pmatrix}     
 \bm{F}_m & \bm{G}_m\\
 \bm{G}_m^\top & \bm{F}_m,
\end{pmatrix},
\end{equation}
where $\bm{F}_m$ and $\bm{G}_m$ are matrices with entries $f_m(n_i,n_j)$ and $g_m(n_i,n_j)$, respectively. Hence, $\bm{F}_m$ is a positive semidefinite matrix, whereas $\bm{G}_m$ is an antisymmetric matrix, in such a way that the block matrix (\ref{big-matrix}) is always positive semidefinite. The covariance function of $Z(L,\ell)$, taking into account conditions (C1)-(C3), is given by
\begin{multline}
\label{cov_as2}
C(L_1,L_2,\Delta \ell)     =   \sum_{n,n'=0}^\infty  f_0(n,n') \widetilde{P}_{n0}(\cos L_1) \widetilde{P}_{n'0}(\cos L_2)    \\
+  2 \sum_{m=1}^\infty  \sum_{n,n'=m}^\infty  \bigg\{   f_m(n,n') \cos(m\Delta \ell)  + g_m(n,n') \sin(m\Delta \ell)  \bigg\}   \widetilde{P}_{nm}(\cos L_1) \widetilde{P}_{n'm}(\cos L_2),
\end{multline}
where the summability condition
\begin{description}
\item[{(C4)}] $ \displaystyle \sum_{n,n'=0}^\infty  f_0(n,n') \widetilde{P}_{n0}(\cos L) \widetilde{P}_{n'0}(\cos L)
+  2 \sum_{m=1}^\infty  \sum_{n,n'=m}^\infty     f_m(n,n') \widetilde{P}_{nm}(\cos L) \widetilde{P}_{n'm}(\cos L) < \infty,$
\end{description}
for all $L\in[0,\pi]$, ensures that $Z(L,\ell)$ is a well-defined second-order process. Note that a zero-mean axially symmetric Gaussian process is completely characterized by $f_m(n,n')$ and $g_m(n,n')$.

Before concluding this section, the following comments are provided.
\begin{itemize}

\item[$-$] When $f_{m}(n,n') = \delta_n^{n'} \xi_n$, for some sequence $\{\xi_n: n\in\mathbb{N}_0\}$ of nonnegative real numbers, and $g_m(n,n')$ is identically equal to zero, we obtain an isotropic process. In this special case, the addition theorem for spherical harmonic functions implies that
\begin{equation}
\label{schoenberg1}
C(L_1,L_2,\Delta \ell)    =   \sum_{n=0}^\infty   { \frac{\xi_n(2n+1)}{4\pi} }  P_n\left(\cos  \text{d}_{\text{GC}}\left( L_1,L_2,\Delta \ell\right) \right),
\end{equation}
where the condition for finite variance is
\begin{equation}
\label{sum_cond}
{  \sum_{n=0}^\infty \xi_n(2n+1) < \infty.  }
 \end{equation}
Condition (\ref{sum_cond}) together with the inequality $|P_n(t)| \leq 1$, for $|t| \leq 1$, imply the uniform convergence of (\ref{schoenberg1}) and ensure its continuity (this is a consequence of Weierstrass M-test). Equation (\ref{schoenberg1}) is the characterization of any continuous covariance function associated with an isotropic process on $\mathbb{S}^2$ (\citealp{schoenberg1942}).

\item[$-$] Suppose now that the coefficients in (\ref{cov_as2}) vanish for $m>0$; then, the covariance function has a longitudinally independent structure because under this choice, we eliminate the terms depending on $\Delta \ell$.

\item[$-$]  When $g_m(n,n')$ is different from zero, we obtain a longitudinally irreversible process. Thus, the function $g_m(n,n')$ is responsible for the asymmetry of the covariance function along longitudes.

\end{itemize}

\section{Main Results}

\label{main}

\subsection{A Bridge Between Isotropy and Longitudinal Independence}

In this section, inspired by the work of \cite{Emery2019}, we describe how to obtain a unified representation of the limit cases, isotropy and longitudinal independence, by means of a parametric family of coefficients $f_m(n,n')$. The antisymmetric part $g_m(n,n')$ will be analyzed in the next subsection, so for the moment we assume that it is identically equal to zero.

Consider the covariance function (\ref{cov_as2}) with coefficients of the form
\begin{equation}
\label{spectrum}
f_{m}(n,n') = \sqrt{\xi_n \xi_{n'} } \rho(n-n') \lambda_m,
\end{equation}
where $\{\xi_n: n\in\mathbb{N}_0\}$ is a nonnegative sequence, $\rho$ is a stationary correlation function, and $\{\lambda_m: m\in\mathbb{N}_0\}$ is a nonnegative and bounded sequence. For every fixed $m\in\mathbb{N}_0$, (\ref{spectrum}) is clearly a positive semidefinite function.

We restrict attention to the following limit cases:
\begin{itemize}
\item[$-$] Suppose first that $\lambda_m = 1$, for all $m\in\mathbb{N}_0$. Thus, as the range of $\rho$ decreases to zero, $f_{m}(n,n')$ goes to $\delta_n^{n'} \xi_n$, converging to the isotropic case.
\item[$-$]  Longitudinal independence arises when $\lambda_0 > 0$ and $\lambda_m = 0$ for every $m\geq 1$, regardless of the choice of $\rho$.
\end{itemize}

There exist various ways to construct a sequence $\{\lambda_m: m\in\mathbb{N}_0\}$ that unifies both cases. For instance, the sequence could be taken as $\lambda_m = 1_{[0,\alpha]}(m)$, where $\alpha \in \mathbb{N}_0$ is a parameter and $1_A$ denotes the indicator function of $A$. While $\alpha=0$ represents longitudinal independence, $\alpha \rightarrow \infty$ corresponds to isotropy (provided that $\rho$ is a Kronecker delta). Another interesting alternative is $\lambda_m = (1 + \gamma m^2)^{-1}$, where $\gamma \geq 0$ is a continuous parameter. Here, $\gamma = 0$ and $\gamma \rightarrow \infty$ correspond to isotropy and longitudinal independence, respectively. Any isotropic model introduced in the literature can be obtained as a special case of (\ref{spectrum}). In this work, we prefer the sparse structure $\lambda_m = 1_{[0,\alpha]}(m)$ because it provides computational advantages. Series expansions with a large proportion of zeros have also been applied by \cite{stein2007spatial} in the study of total column ozone.

To ensure that (\ref{spectrum}) yields a well-defined second-order Gaussian process, condition (C4) must be verified. The following proposition
shows that under adequate assumptions of the asymptotic decay of $\{\xi_n: n\in\mathbb{N}_0\}$, condition (C4) holds.
\begin{proposition}
\label{sumabilidad_prop}
Consider $f_m(n,n')$ as in (\ref{spectrum}), and suppose that, for $n > n_0$,  $$\xi_n \leq r n^{-\beta},$$
for some constants $r>0$ and $n_0\in\mathbb{N}$. Then, condition (C4) holds if either
\begin{itemize}
\item[(i)] $\beta>4$; or
\item[(ii)] $\beta>2$ and $\rho(h)  = \delta_0^h$.
\end{itemize}
\end{proposition}
\begin{proof}
Using the inequality \citep{doi:10.1002/sapm195534143}
$$  \max_{x \in [-1,1]} |\widetilde{P}_{nm}(x)| \leq  \sqrt{2n+1},$$
we observe that a sufficient condition to obtain a process with finite variance is
$$\displaystyle \sum_{n,n'=0}^\infty   f_0(n,n') \sqrt{(2n+1)(2n'+1)}
+  2 \sum_{m=1}^\infty  \sum_{n,n'=m}^\infty    f_m(n,n') \sqrt{(2n+1)(2n'+1)}  < \infty.$$
We only analyze the terms with $m>n_0$ (the other terms, for fixed $m\leq n_0$, can be studied in a similar manner). Since $\rho(h) \leq 1$, for all $h$, and $\{\lambda_m: m\in\mathbb{N}_0\}$ is a nonnegative and bounded sequence, we have
\begin{eqnarray}
\label{condicion_final}
\displaystyle  
\sum_{m>n_0}^\infty  \sum_{n,n'=m}^\infty   f_m(n,n')  \sqrt{(2n+1)(2n'+1)}   &    \leq  &  \sum_{m>n_0}^\infty \lambda_m  \sum_{n,n'=m}^\infty    \sqrt{\xi_n \xi_{n'}   {(2n+1)(2n'+1)} }\\ \nonumber
  &   \leq     &   \left(\sup_m \lambda_m\right)      \sum_{m>n_0}^\infty  \left( \sum_{n=m}^\infty  \sqrt{ \xi_n(2n+1)}  \right)^2\\  \nonumber
    &   \leq   &    \left(\sup_m \lambda_m\right)  3r \sum_{m>n_0}^\infty  \left( \sum_{n=m}^\infty   n^{-(\beta-1)/2}  \right)^2.  
\end{eqnarray}
Employing an integral bound, one has
\begin{equation*}
 \sum_{n=m}^\infty   n^{-(\beta- 1)/2}   
\leq   m^{-(\beta- 1)/2} +  \int_{m}^\infty   x^{-(\beta- 1)/2} \text{d}x  
 =  m^{-(\beta- 1)/2} +   \frac{2m^{-(\beta - 3)/2}}{\beta - 3}, 
\end{equation*}
for all $\beta>3$. Thus, we have that
$$ m \mapsto \left( \sum_{n=m}^\infty   n^{-(\beta- 1)/2} \right)^2$$ decays algebraically with order $\beta - 3$. We conclude that (\ref{condicion_final}) is finite provided that $\beta > 4$. The first part of the proof is completed.

Suppose now that $\rho(h)  = \delta_0^h$. To ensure (C4), we must verify that
$$ \displaystyle \sum_{n=0}^\infty  \xi_n \lambda_0 \widetilde{P}_{n0}(\cos L) \widetilde{P}_{n0}(\cos L)
+  2 \sum_{m=1}^\infty  \lambda_m \sum_{n=m}^\infty     \xi_n  \widetilde{P}_{nm}(\cos L) \widetilde{P}_{nm}(\cos L) < \infty.$$
Again, using that $\{\lambda_m: m\in\mathbb{N}_0\}$ is bounded, it is sufficient to show that
\begin{equation*}
\displaystyle \sum_{n=0}^\infty  \xi_n \widetilde{P}_{n0}(\cos L) \widetilde{P}_{n0}(\cos L)    
+  2 \sum_{m=1}^\infty  \sum_{n=m}^\infty     \xi_n  \widetilde{P}_{nm}(\cos L) \widetilde{P}_{nm}(\cos L)  < \infty,
\end{equation*}
which is equivalent to
\begin{equation*}
\label{eq_alternativa}
  \sum_{n=0}^\infty  \xi_n  \left\{  \widetilde{P}_{n0}(\cos L) \widetilde{P}_{n0}(\cos L)    
+  2  \sum_{m=1}^n   \widetilde{P}_{nm}(\cos L) \widetilde{P}_{nm}(\cos L)  \right\} < \infty. 
\end{equation*}
From the addition theorem for spherical harmonic functions, we obtain the simplified condition (\ref{sum_cond}), which is true provided that $\beta > 2$.
\end{proof}

The second part of Proposition \ref{sumabilidad_prop}, i.e., when $\rho(h)  = \delta_0^h$, matches previous literature related to the isotropic case \citep{schoenberg1942,lang2015isotropic}. However, it is slightly more general because we are not necessarily assuming that $\{\lambda_m: m\in\mathbb{N}_0\}$ is a sequence of ones.

\subsection{Modeling the Antisymmetric Part}

We illustrate a simple approach to construct the antisymmetric coefficients $g_m(n,n')$ from the coefficients $f_m(n,n')$.

\begin{proposition}
For each $m\in\mathbb{N}$, consider a covariance function $f_m(n,n')$ of type (\ref{spectrum}). Thus, for all $\kappa \in \mathbb{R}$, the cross-covariance function
\begin{equation*}
g_m(n,n') =  \frac{\sqrt{\xi_n \xi_{n'} } \lambda_m}{{4}} \bigg\{\rho(n-n' - \kappa) - \rho(n-n'+\kappa)\bigg\},
\end{equation*}
is an admissible model,  in the sense that (\ref{big-matrix}) is a positive semidefinite matrix.
\end{proposition}
\begin{proof}
 Consider two independent sequences of zero-mean random variables, $\{\widetilde{a}_{nm}\}$ and $\{\widetilde{b}_{nm}\}$, with a covariance structure of type (\ref{spectrum}). For a nonnegative integer $q$, new sequences of coefficients are defined
\begin{equation}
\label{transformation}
\begin{cases}
a_{nm} = \displaystyle \frac{1}{\sqrt{2}}\left(\widetilde{a}_{nm} + \sqrt{\frac{\xi_n}{\xi_{(n+q)}}} \widetilde{b}_{(n+q)m} \right),\\ \\
b_{nm} = \displaystyle \frac{1}{\sqrt{2}}\left( \sqrt{\frac{\xi_n}{\xi_{(n+q)}}} \widetilde{a}_{(n+q)m} - \widetilde{b}_{nm}\right).
\end{cases}
\end{equation}
Hence, $\text{cov}\{ {a}_{nm}, {a}_{n'm'} \}$ and $\text{cov}\{ {b}_{nm}, {b}_{n'm'} \}$ are given by (\ref{spectrum}), that is, the marginal covariance functions are preserved by transformation (\ref{transformation}). In addition, the cross-covariance function between the sequences $\{{a}_{nm}\}$ and $\{{b}_{nm}\}$ is given by
\begin{equation}
\label{parte-antisimetrica}
\text{cov}\{ {a}_{nm}, {b}_{n'm'} \} = -\text{cov}\{ {b}_{nm}, {a}_{n'm'} \} =   \frac{\delta_{m}^{m'} \sqrt{\xi_n \xi_{n'} } \lambda_m}{{4}} \bigg\{\rho(n-n' - q) - \rho(n-n'+q)\bigg\}.
\end{equation}
 Following a similar scheme with continuous indices, we conclude that the cross-covariance function (\ref{parte-antisimetrica}) is also valid if we replace the integer $q$ with a parameter $\kappa$ that varies continuously on $\mathbb{R}$. More precisely, we consider random variables $\widetilde{a}_{nm}$ and $\widetilde{b}_{nm}$ with indices $n$ and $m$ on the real line. The sequences ${a}_{nm}$ and ${b}_{nm}$ are constructed as in (\ref{transformation}) (with $\kappa$ instead of $q$) by restricting these indices to be nonnegative integers. The obtained sequences will have the desired cross-covariance function with $\kappa\in\mathbb{R}$.
\end{proof}

The antisymmetric part is identically equal to zero when $\kappa = 0$, recovering a longitudinally reversible process. In Section \ref{ilustracion-antisymmetric}, we will illustrate the impact of $g_m(n,n')$ on the covariance function as well as on the realizations of the process.

\subsection{Finite Karhunen-Lo\`eve Expansion and its $L^2$-Error}

This section focuses on the approximation of axially symmetric Gaussian processes through a finite linear combination of spherical harmonic functions. This approximation can be performed by means of a truncated version of (\ref{proceso}), where truncation is taken with respect to index $n$. Specifically, given a large $N\in\mathbb{N}$, we consider
\begin{equation}
\label{sim}
\widehat{Z}_N(L,\ell) =    \sum_{n=0}^{N}   a_{n0} \widetilde{P}_{n0}(\cos L)  + 2 \sum_{m=1}^N    \sum_{n=m}^N   \bigg\{ a_{nm} \cos(m\ell)+ b_{nm} \sin(m\ell) \bigg\} \widetilde{P}_{nm}(\cos L).
\end{equation}
Note that (\ref{sim}) is a truncation of (\ref{proceso}), where we have simply modified the order of summation for convenience. This approximation technique was developed for isotropic spatial processes by \cite{lang2015isotropic} and for spatially isotropic space-time processes by \cite{de2018regularity}.

The purpose now is to derive a bound for the $L^2$-error in terms of $N$. The following proposition characterizes the accuracy of the approximation.

\begin{proposition}
\label{prop1}
Let $N\in\mathbb{N}$ and consider the stochastic processes $Z(L,\ell)$ and $\widehat{Z}_N(L,\ell)$ in (\ref{proceso}) and (\ref{sim}), respectively, under conditions (C1)-(C4). Then,
\begin{equation}
\| Z-\widehat{Z}_N \|^2_{L^2(\Omega \times \mathbb{S}^2)}   =   \sum_{n=N+1}^\infty   f_0(n,n) + 2\sum_{n=N+1}^\infty \sum_{m=1}^n  f_m(n,n).
\end{equation}
\end{proposition}
\begin{proof}
First, observe that $Z-\widehat{Z}_N$ can be split into three parts, namely, $Z -\widehat{Z}_N =   T_1 + 2 T_2 + 2 T_3$, where
$$T_j =   \sum_{(n,m)\in \Delta_j}   \bigg\{ a_{nm} \cos(m\ell)+ b_{nm} \sin(m\ell) \bigg\} \widetilde{P}_{nm}(\cos L),$$
with $\Delta_1 = \{ (n,m): n > N \text{ and } m=0 \}$, $\Delta_2 = \{ (n,m): n > N \text{ and } 1 \leq m \leq N  \}$, and $\Delta_3 = \{ (n,m): n \geq  m  \text{ and } m > N  \}$. According to conditions (C1)-(C4), $T_i$ and $T_j$ are uncorrelated processes for all $i\neq j$ since the sets $\Delta_i$ and $\Delta_j$ are disjoint sets in the index $m$. Thus,
$$\| Z-\widehat{Z}_N \|^2_{L^2(\Omega \times \mathbb{S}^2)} =  \| T_1\|^2_{L^2(\Omega\times \mathbb{S}^2)} + 4\| T_2\|^2_{L^2(\Omega\times \mathbb{S}^2)} + 4\| T_3\|^2_{L^2(\Omega\times \mathbb{S}^2)}.$$
The decomposition of $Z-\widehat{Z}_N$ into three mutually uncorrelated processes is a key part of this proof. Using integration in terms of spherical coordinates, one has
\begin{equation*}
 \| T_1\|^2_{L^2(\Omega\times \mathbb{S}^2)} =  \sum_{n,n'=N+1}^\infty E\{a_{n0}a_{n'0}\} \int_0^{2\pi} \int_0^{\pi} \widetilde{P}_{n0}(\cos L) \widetilde{P}_{n'0}(\cos L) \sin L \, \text{d}L \, \text{d}\ell.
\end{equation*}
Recall that $E\{a_{n0}a_{n'0}\} = f_0(n,n')$. Additionally, using the orthogonality properties of the associated Legendre polynomials (see \citealp{abramowitz1988handbook}),
\begin{equation}
\label{ortogonalidad}
 \int_0^{\pi} \widetilde{P}_{nm}(\cos L) \widetilde{P}_{n'm}(\cos L) \sin L \, \text{d}L = \frac{\delta_n^{n'}}{2\pi},
  \end{equation}
we conclude that
$  \| T_1\|^2_{L^2(\Omega\times \mathbb{S}^2)} =  \sum_{n=N+1}^\infty  f_0(n,n).$
In addition, we have
\begin{multline*}
\| T_2\|^2_{L^2(\Omega\times \mathbb{S}^2)} =   \int_0^{2\pi}   \sum_{m,m'=1}^N  \sum_{n,n'=N+1}^\infty   E\bigg(   \bigg\{ a_{nm} \cos(m\ell)+ b_{nm} \sin(m\ell) \bigg\}  \\   \times \bigg\{ a_{n'm'} \cos(m'\ell)+ b_{n'm'} \sin(m'\ell) \bigg\}  \bigg)
\times \int_0^{\pi}  \widetilde{P}_{nm}(\cos L)   \widetilde{P}_{n'm'}(\cos L) \sin L \, \text{d}L \, \text{d}\ell .
\end{multline*}
Using conditions (C1)-(C4) and the orthogonality properties of the trigonometric Fourier basis on $[0,2\pi]$, we have
\begin{equation*}
 \| T_2\|^2_{L^2(\Omega\times \mathbb{S}^2)} =     \sum_{m=1}^N  \sum_{n,n'=N+1}^\infty \pi  f_m(n,n')     \int_0^{\pi}   \widetilde{P}_{nm}(\cos L)   \widetilde{P}_{n'm}(\cos L) \sin L \, \text{d}L.
\end{equation*}
Thus,   (\ref{ortogonalidad}) implies that
$$ \| T_2\|^2_{L^2(\Omega\times \mathbb{S}^2)} =      \sum_{m=1}^N \sum_{n=N+1}^\infty \frac{f_m(n,n)}{2}.$$
Using similar arguments, we obtain
$$ \| T_3\|^2_{L^2(\Omega\times \mathbb{S}^2)} =     \sum_{m=N+1}^\infty    \sum_{n=m}^\infty \frac{f_m(n,n)}{2}.$$
We conclude the proof by noting that
$$ \| T_2\|^2_{L^2(\Omega\times \mathbb{S}^2)}  +  \| T_3\|^2_{L^2(\Omega\times \mathbb{S}^2)} = \sum_{n=N+1}^\infty \sum_{m=1}^n\frac{f_m(n,n)}{2}.$$

\end{proof}

A remarkable feature of Proposition \ref{prop1} is that the $L^2$-error is completely characterized by the decay of the diagonal elements of the matrices $\bm{\Gamma}_m$ in (\ref{big-matrix}). Therefore, the antisymmetric part does not influence this quantity. In the isotropic case, we obtain a corollary result previously reported by \cite{lang2015isotropic}.

\begin{cor}
\label{prop2}
Let $N\in\mathbb{N}$ and consider $Z(L,\ell)$ and $\widehat{Z}_N(L,\ell)$ in (\ref{proceso}) and (\ref{sim}), respectively, under conditions (C1)-(C4). In addition, suppose that $f_m(n,n') = \delta_n^{n'} \xi_n$. Then,
\begin{equation}
\| Z-\widehat{Z}_N \|^2_{L^2(\Omega \times \mathbb{S}^2)}   =   \sum_{n=N+1}^\infty (2n+1)  \xi_n.
\end{equation}
\end{cor}

Since Proposition \ref{prop1} gives the $L^2$-error for every axially symmetric process, we derive an explicit bound when the coefficients $f_m(n,n')$ are of the form (\ref{spectrum}). Such a result is reported in the following corollary.

\begin{cor}
\label{bound-particular-spectrum}
Let $N\in\mathbb{N}$ and consider $Z(L,\ell)$ and $\widehat{Z}_N(L,\ell)$ in (\ref{proceso}) and (\ref{sim}), respectively, under conditions (C1)-(C4) with $f_m(n,n')$ of the form (\ref{spectrum}). Suppose that there exist constants $\beta>2$, $r>0$ and $n_0\in\mathbb{N}$ such that $\xi_n \leq r n^{-\beta}$, for all $n > n_0$. Thus,
\begin{equation}
\| Z-\widehat{Z}_N \|^2_{L^2(\Omega \times \mathbb{S}^2)}   \leq c N^{-(\beta-2)},
\end{equation}
for some constant $c>0$ that depends on $r,\beta,n_0$ and that is independent of $N$.
\end{cor}
Corollary \ref{bound-particular-spectrum} above extends Proposition 5.2 in \cite{lang2015isotropic}, from the isotropic to the axially symmetric case.
When $f_m(n,n')$ is of the form (\ref{spectrum}), we have 
\begin{eqnarray*}
\sum_{n=N+1}^\infty  f_0(n,n)  + 2\sum_{n=N+1}^\infty \sum_{m=1}^n  f_m(n,n)  \leq  \left(  \sup_m \lambda_m \right)   \sum_{n=N+1}^\infty   (2 n+1) \xi_n.
\end{eqnarray*}
Thus, the proof of Corollary \ref{bound-particular-spectrum} follows the same arguments given by \cite{lang2015isotropic}. Additionally, it is worth noting that the same rate of convergence obtained in Corollary \ref{bound-particular-spectrum} can be derived in terms of $L^p$ norms for every $p>0$. The proof of this assertion is completely analogous to the proofs reported in \cite{lang2015isotropic} and \cite{cleanthous2020regularity}, so it is omitted.

\section{Numerical Experiments}
\label{simulacion}

\subsection{Simulating Axially Symmetric Gaussian Processes}

The approximation methodology developed in the previous section suggests a natural simulation algorithm based on a weighted sum of finitely many spherical harmonic functions. We simulate longitudinally reversible Gaussian processes (antisymmetric coefficients are explored in the next subsection). Consider the following particular cases of (\ref{spectrum}):

\begin{description}

\item[\bf Example 1.] Let $\rho(h) = \delta_0^{h}$ and consider a sequence $\{\xi_n: n\in\mathbb{N}_0\}$ of Legendre-Mat\'ern type \citep{guinness2016isotropic}, that is,
\begin{equation*}
\label{lm_seq}
\xi_n = (\tau^2 + n^2)^{-\nu-1/2},
\end{equation*}
where $\tau$ and $\nu$ are positive parameters. The sequence $\{\lambda_m: m\in\mathbb{N}_0\}$ is taken as $\lambda_m = 1_{[0,\alpha]}(m)$. When $\alpha \rightarrow \infty$, the covariance function of the process converges to the Legendre-Mat\'ern model of \cite{guinness2016isotropic}, which is associated with an isotropic process for which $\tau$ and $\nu$ regulate the range and mean square differentiability of the sample paths, respectively (see \citealp{guinness2016isotropic} for details). Recall that a longitudinally independent structure is obtained with $\alpha = 0$. Using Proposition \ref{sumabilidad_prop}, it is straightforward to verify that the summability condition (C4) holds when $\nu>1/2$. In Figure \ref{fig_guinness}, we report simulated realizations with $\nu = 1.5$, $\tau^2 = 100$ and different values for $\alpha$. We set $N = 200$ and a grid of longitudes and latitudes of size $500 \times 500$. As $\alpha$ increases, the features of the realizations gradually change from longitudinal independence to isotropy.

\begin{figure}
  \centering
\includegraphics[scale=0.1]{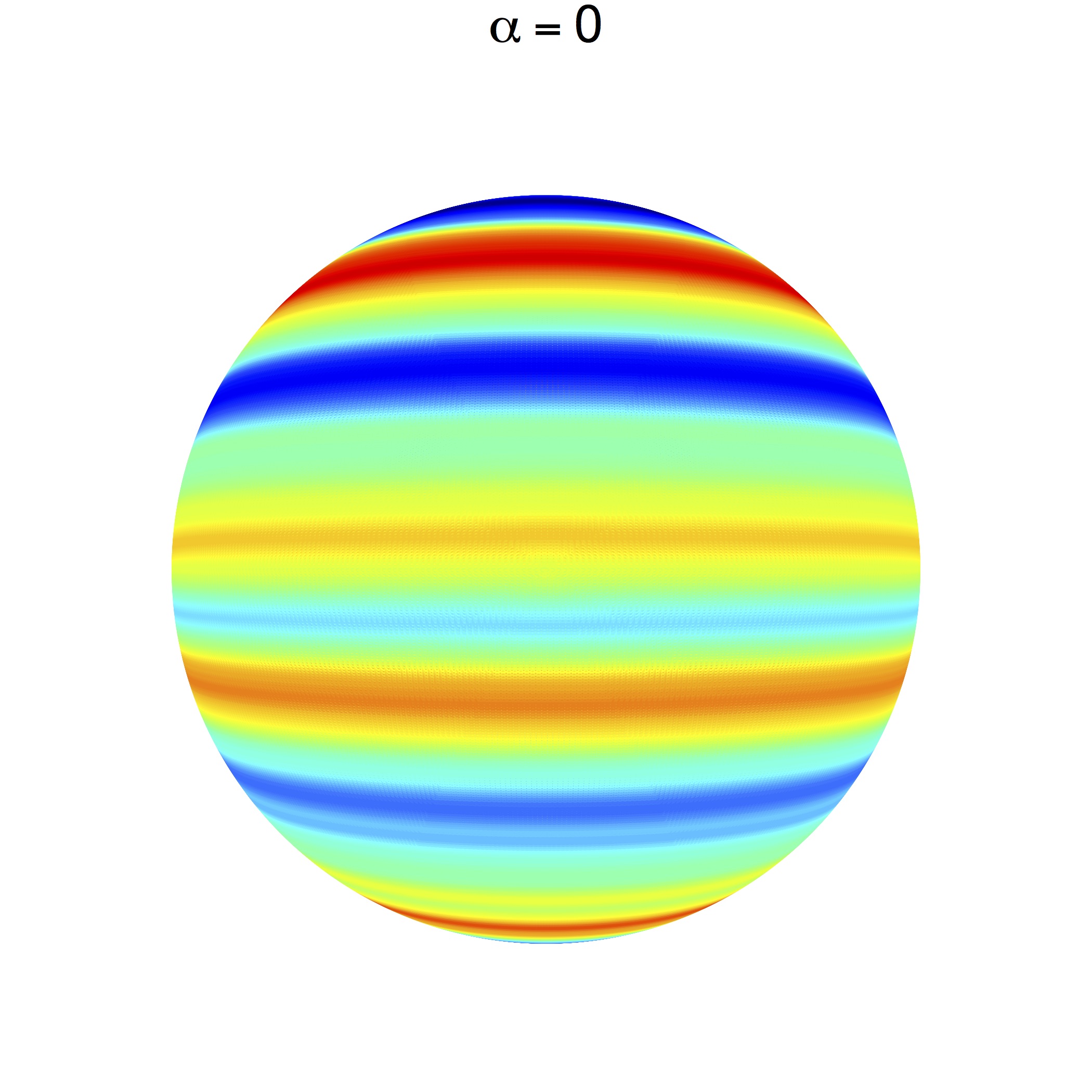}     \includegraphics[scale=0.1]{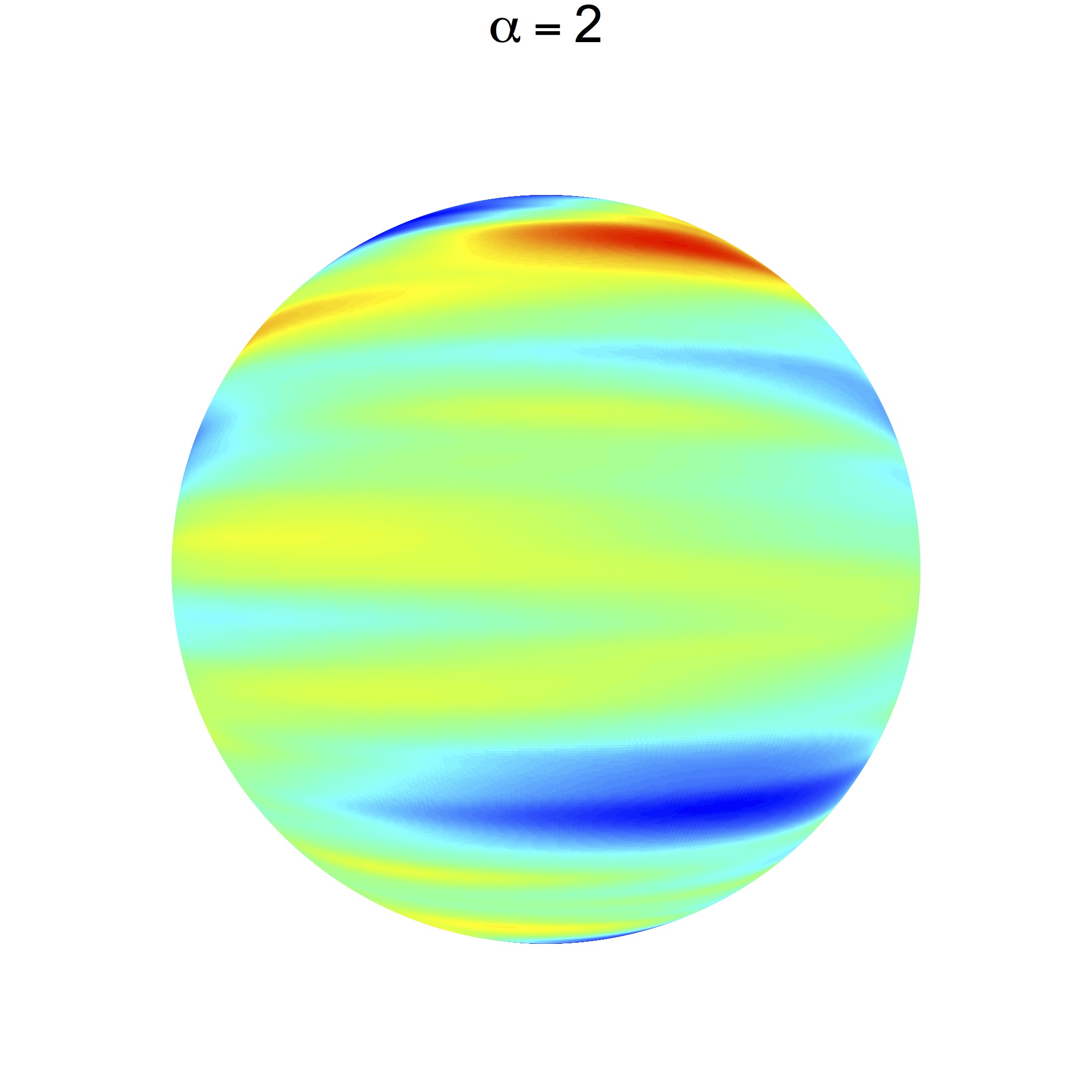}
\includegraphics[scale=0.1]{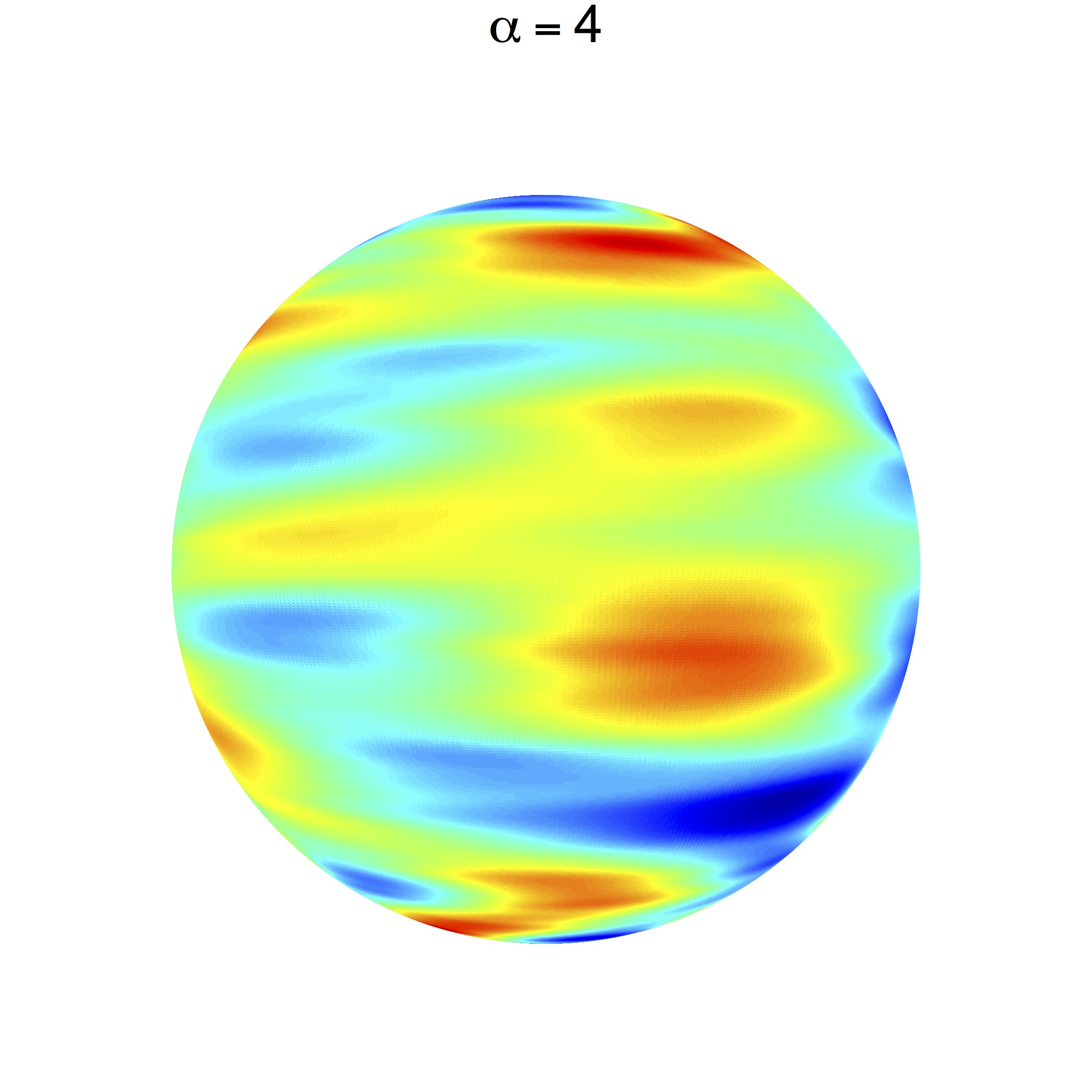}   \includegraphics[scale=0.1]{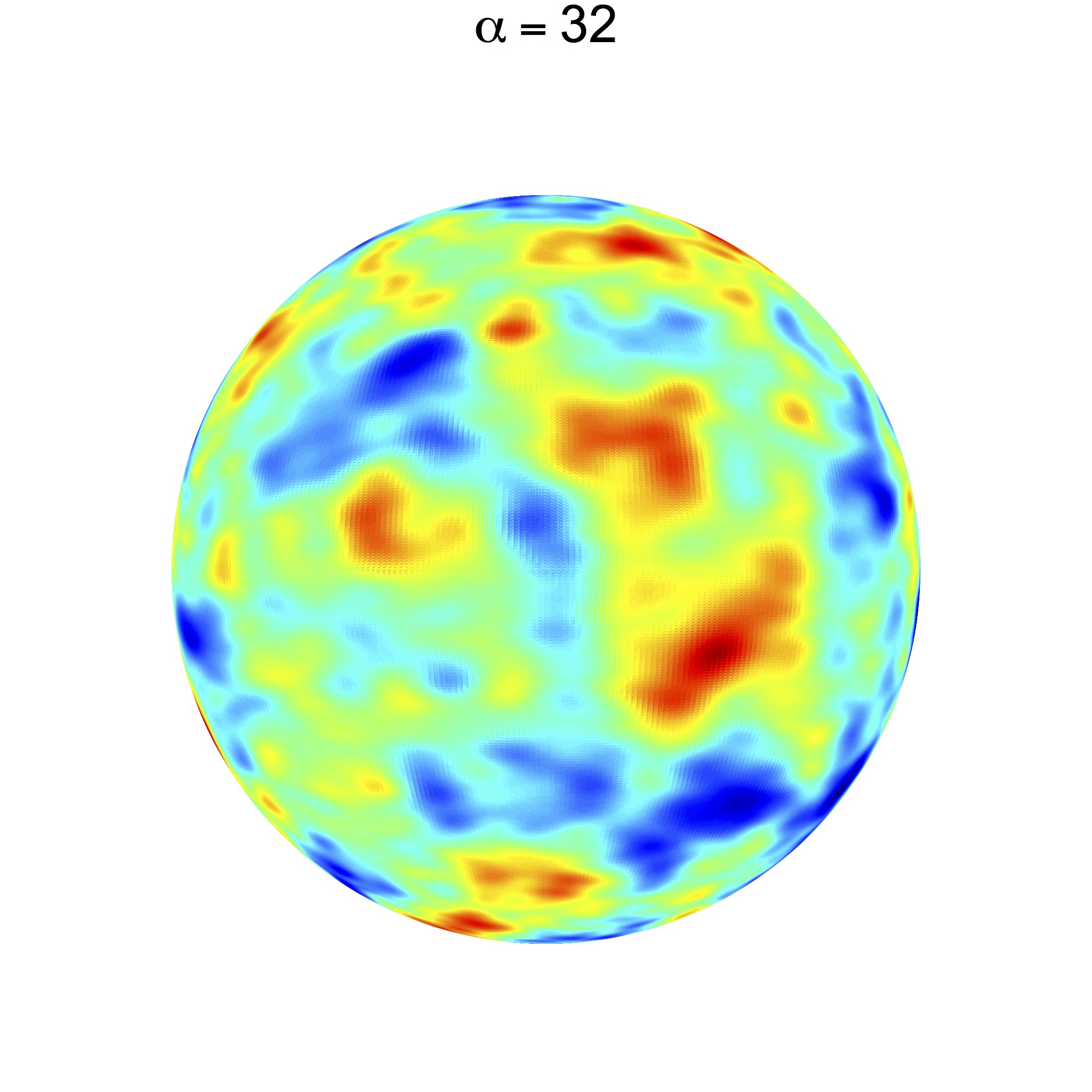}
\caption{Simulated processes on a grid of longitudes and latitudes of size $500\times 500$, with $N=200$, and the covariance function described in Example 1, with $\nu = 1.5$, $\tau^2 = 100$, and different values for $\alpha$. The same random seed has been used for each realization.}
        \label{fig_guinness}
\end{figure}

\item[\bf Example 2.] The multiquadric model (see, e.g., \citealp{gneiting2013strictly}) is characterized by the sequence
\begin{equation}
\label{multiq_seq}
\xi_n = (1-\delta) \delta^n, \qquad n\in\mathbb{N}_0,
\end{equation}
where $0<\delta<1$ is a parameter. The correlation function $\rho$ and the sequence $\{\lambda_m: m\in\mathbb{N}_0\}$ are taken as in Example 1. Here, the summability condition (C4) is visibly satisfied for any $0<\delta<1$. Figure \ref{fig_multiq} displays simulated realizations with $\delta=0.7$ and different values for $\alpha$. We again consider $N = 200$ and a grid of longitudes and latitudes of size $500 \times 500$. The effect of $\alpha$ on the realizations is similar to that reported in Example 1.

\begin{figure}
  \centering
\includegraphics[scale=0.1]{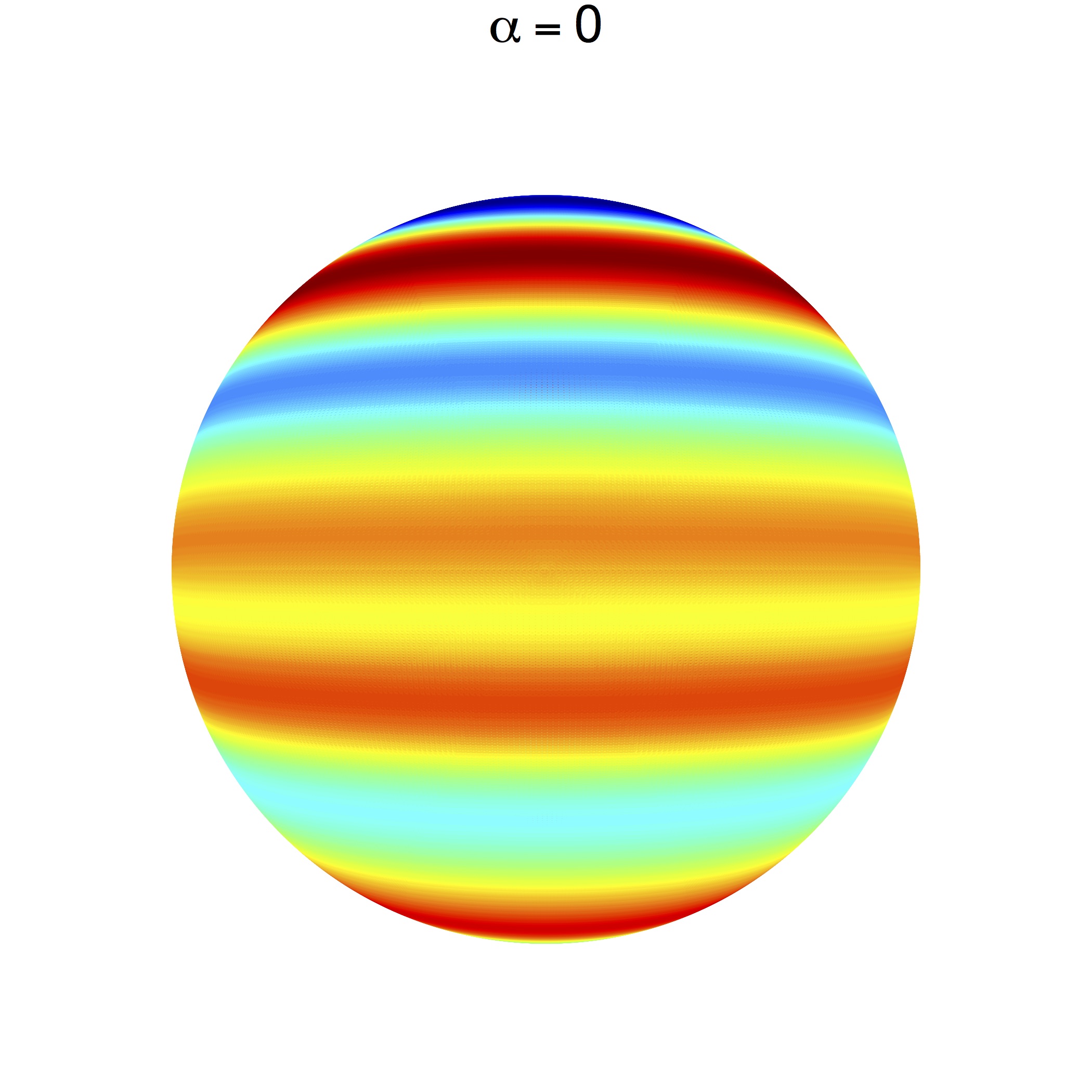}     \includegraphics[scale=0.1]{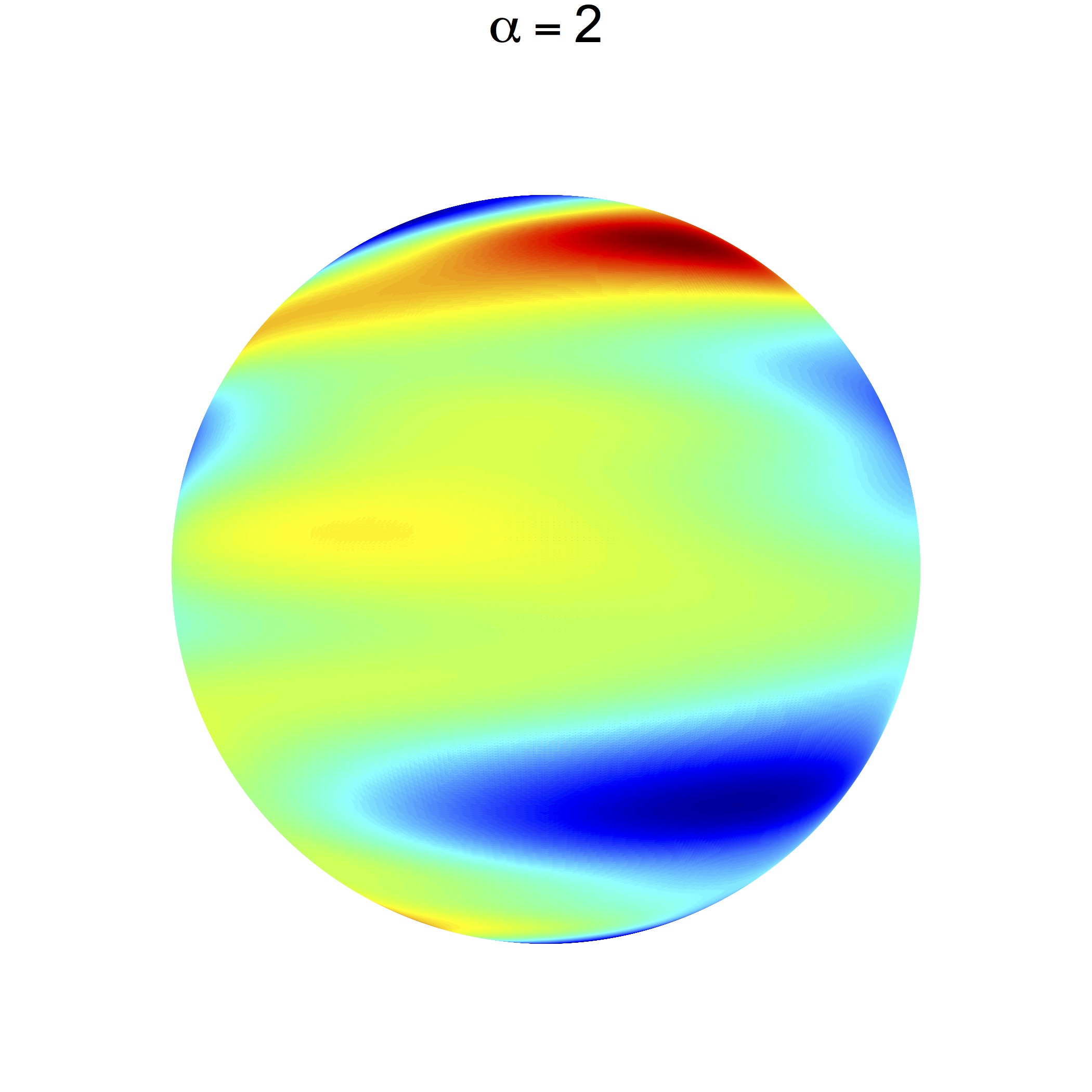}
   \includegraphics[scale=0.1]{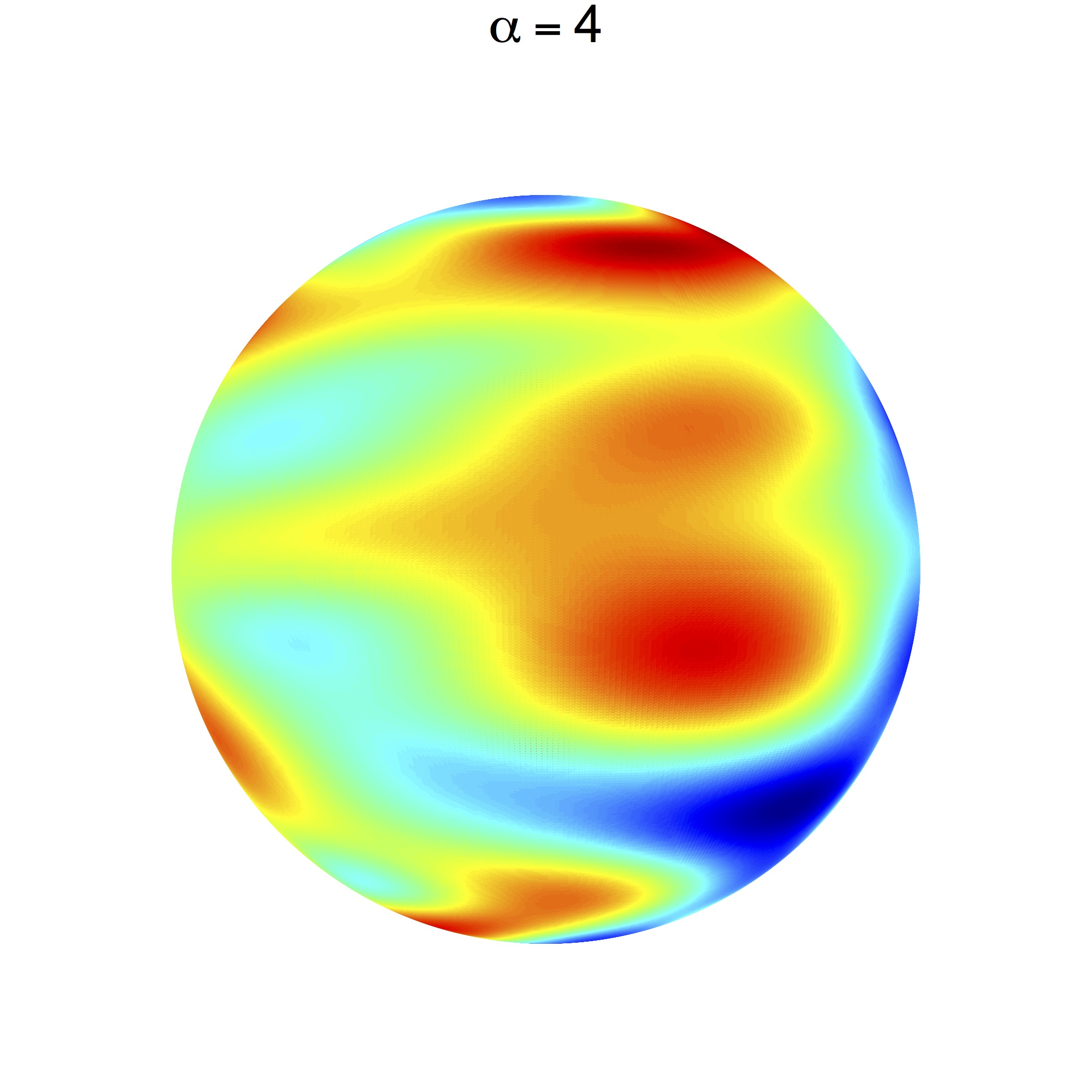}        \includegraphics[scale=0.1]{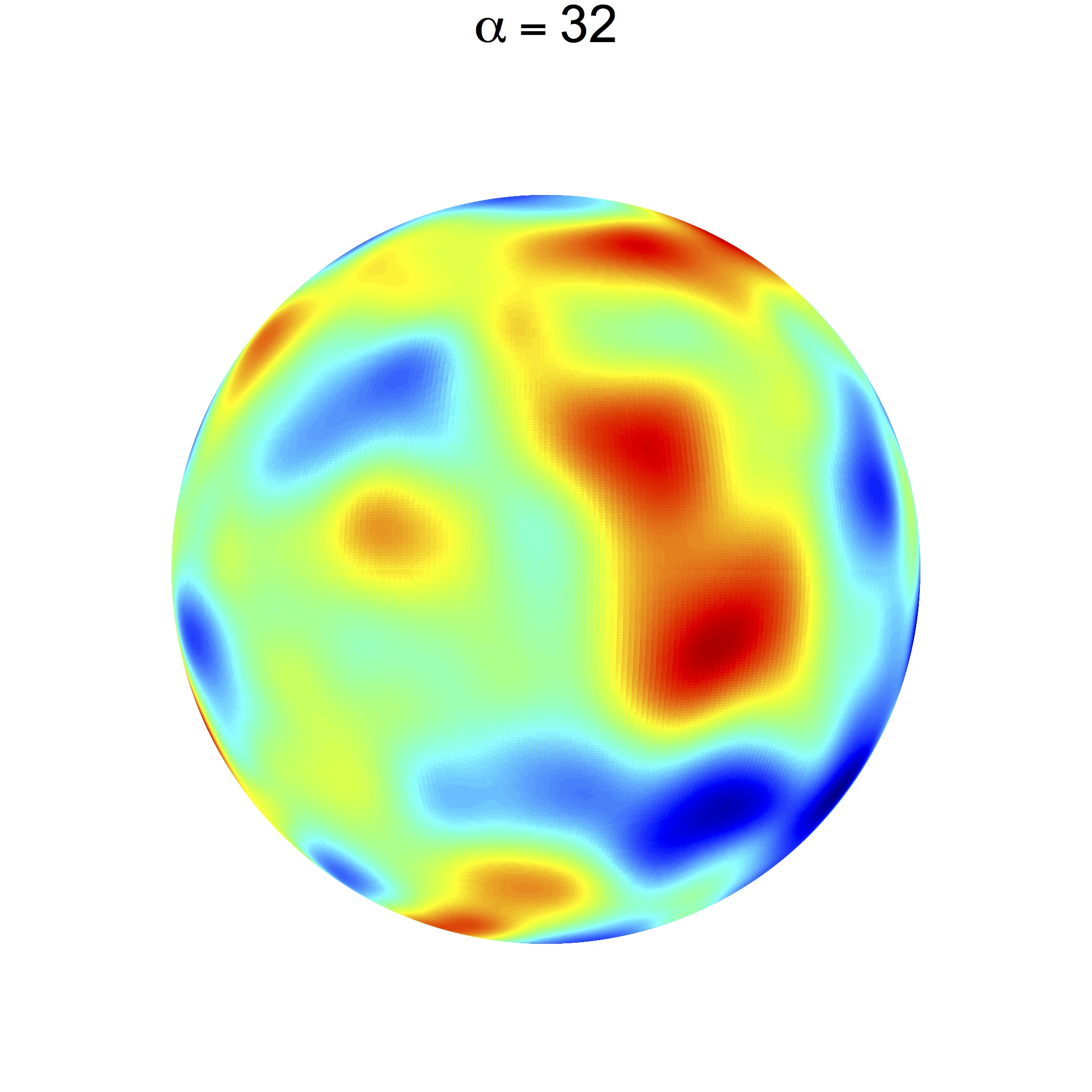}
\caption{Simulated processes on a grid of longitudes and latitudes of size $500\times 500$, with $N=200$, and the covariance function described in Example 2, with $\delta=0.7$, and different values for $\alpha$. The same random seed has been used for each realization.}
        \label{fig_multiq}
\end{figure}

\item[\bf Example 3.] Unlike the previous examples, where $\rho$ has been taken as a Kronecker delta, we now explore the effect of a different correlation on the realizations of the process. We consider the multiquadric coefficients (\ref{multiq_seq}), $\lambda_m = 1_{[0,\alpha]}(m)$ and an exponential correlation $\rho(h) = \exp(-\phi |h|)$, where $\phi$ is a positive parameter. Figure \ref{fig_ej3} reports realizations with $\delta=0.7$, $\alpha = 4$, and different values for $\phi$. Again, we set $N = 200$ and a spatial grid of size $500 \times 500$. As $\phi$ increases, the range of the correlation decreases, and we obtain a behavior that is similar to that reported in Example 2, as expected. However, as $\phi$ decreases, we observe realizations that are characterized by strong correlations along latitudes.

\begin{figure}
  \centering
\includegraphics[scale=0.1]{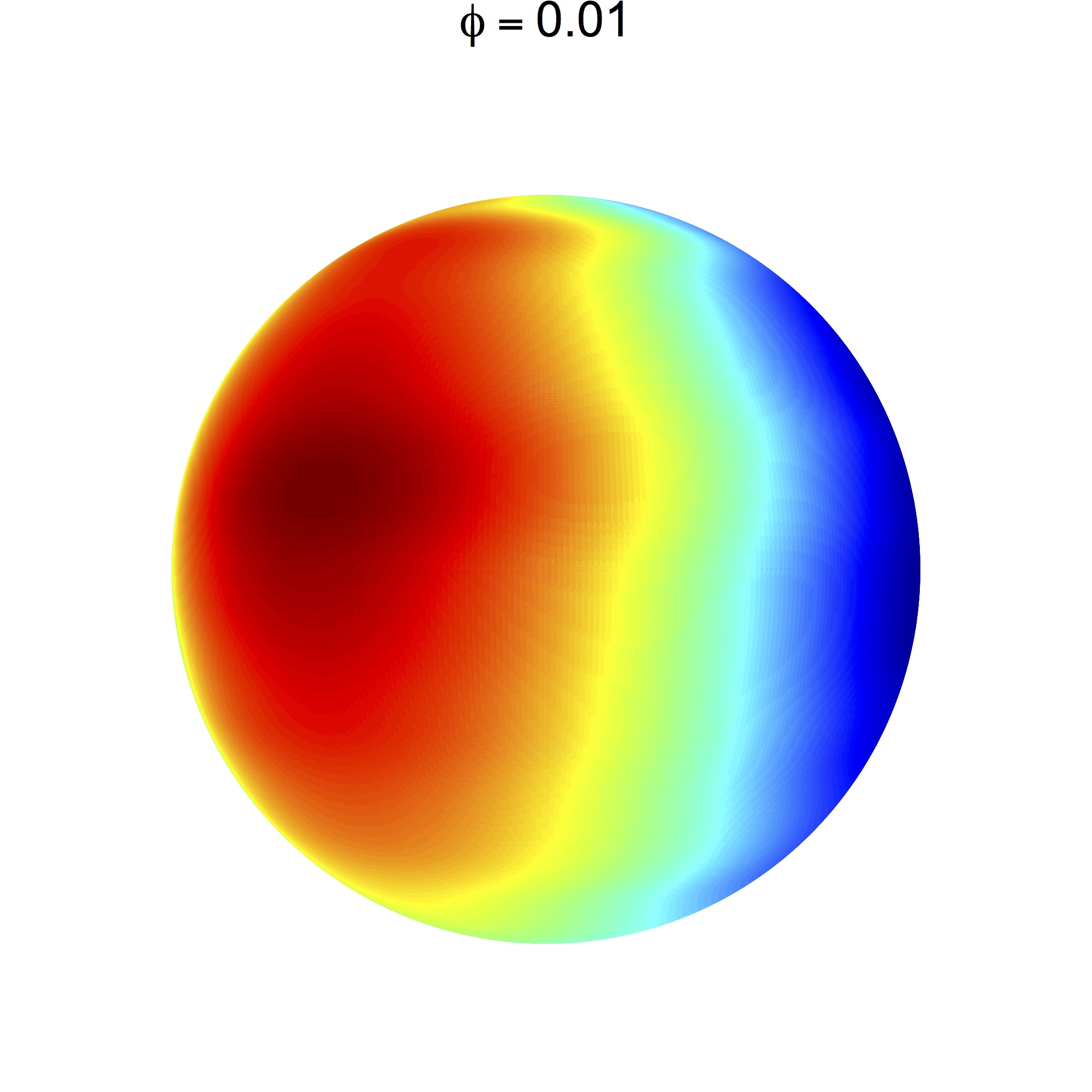}     \includegraphics[scale=0.1]{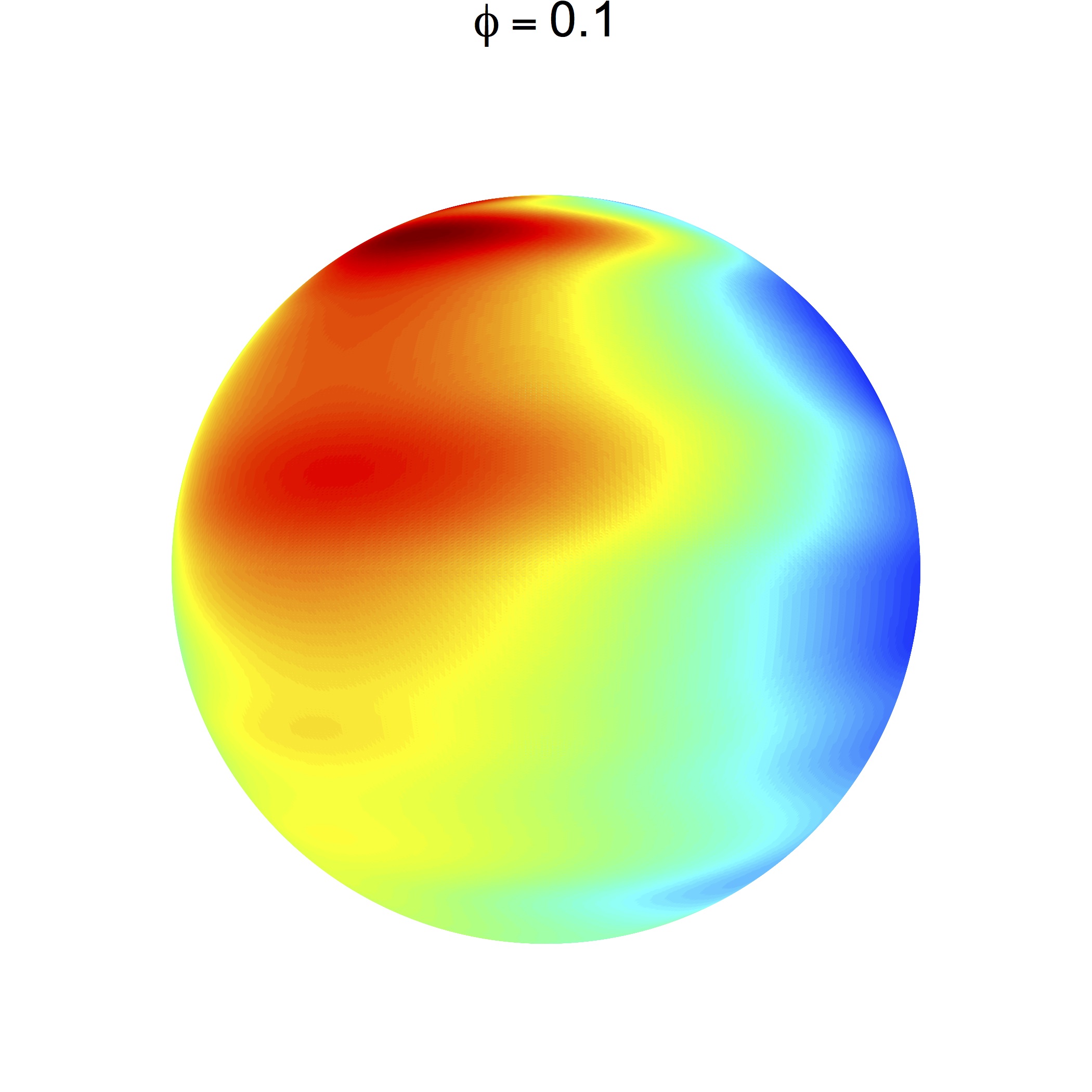}
   \includegraphics[scale=0.1]{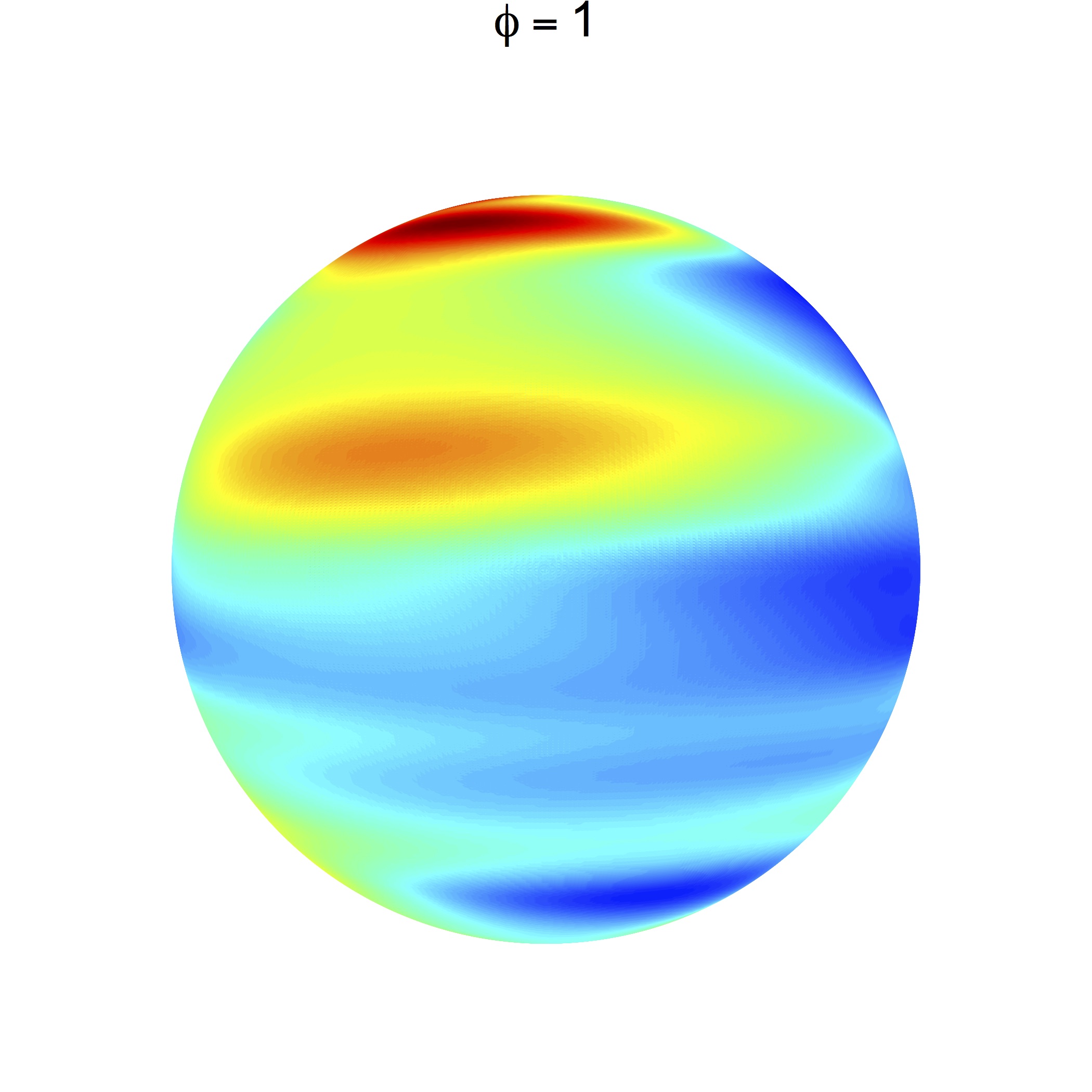}
\caption{Simulated processes on a grid of longitudes and latitudes of size $500\times 500$, with $N=200$, and the covariance function described in Example 3, with $\delta=0.7$, $\alpha=2$, and different values for $\phi$. The same random seed has been used for each realization.}
        \label{fig_ej3}
\end{figure}

\end{description}

For additional examples of sequences $\{\xi_n: n\in\mathbb{N}_0\}$, we refer the reader to \cite{ma2012stationary}, \cite{terdik2015angular} and \cite{leonenko2018estimation}.

\subsection{Illustrating the Effect of the Antisymmetric Part}
\label{ilustracion-antisymmetric}

The aim of this section is to show the impact of $g_m(n,n')$ on the covariance function of an axially symmetric process. Figure \ref{contorno2} shows the contour plots of the covariance function given in Example 1 after adding the antisymmetric part (\ref{parte-antisimetrica}) in terms of $L_2$ and $\Delta\ell$ for different fixed values of $L_1 = \pi/3, \pi/2, 2\pi/3$. We set $\tau^2=100$, $\nu=1.5$ and $\alpha = 8$. Each panel considers $\Delta\ell$ varying in the range $[-0.2,0.2]$, whereas $L_2$ varies in the range $[L_1-0.2,L_1+0.2]$. The first row corresponds to a longitudinally reversible process ($\kappa=0$), whereas the second and third rows depict the distortion produced by $g_m(n,n')$ with $\kappa = 1$ and $\kappa=-1$, respectively. We detect a rotation of the contour plots when the parameter $\kappa$ is different from zero. The orientation of this rotation depends on the sign of $\kappa$.

\begin{figure}
  \centering
\includegraphics[scale=0.25]{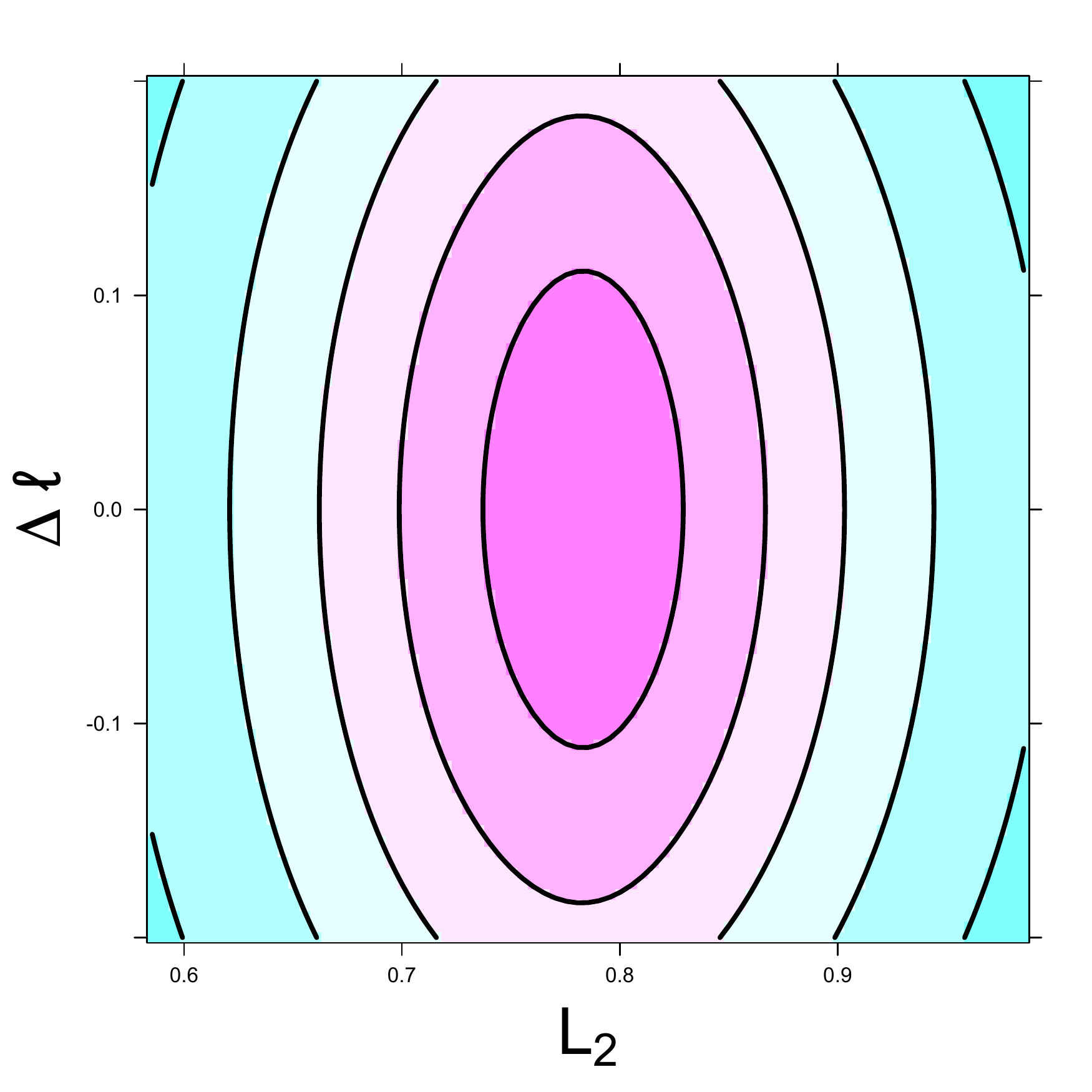}     \includegraphics[scale=0.25]{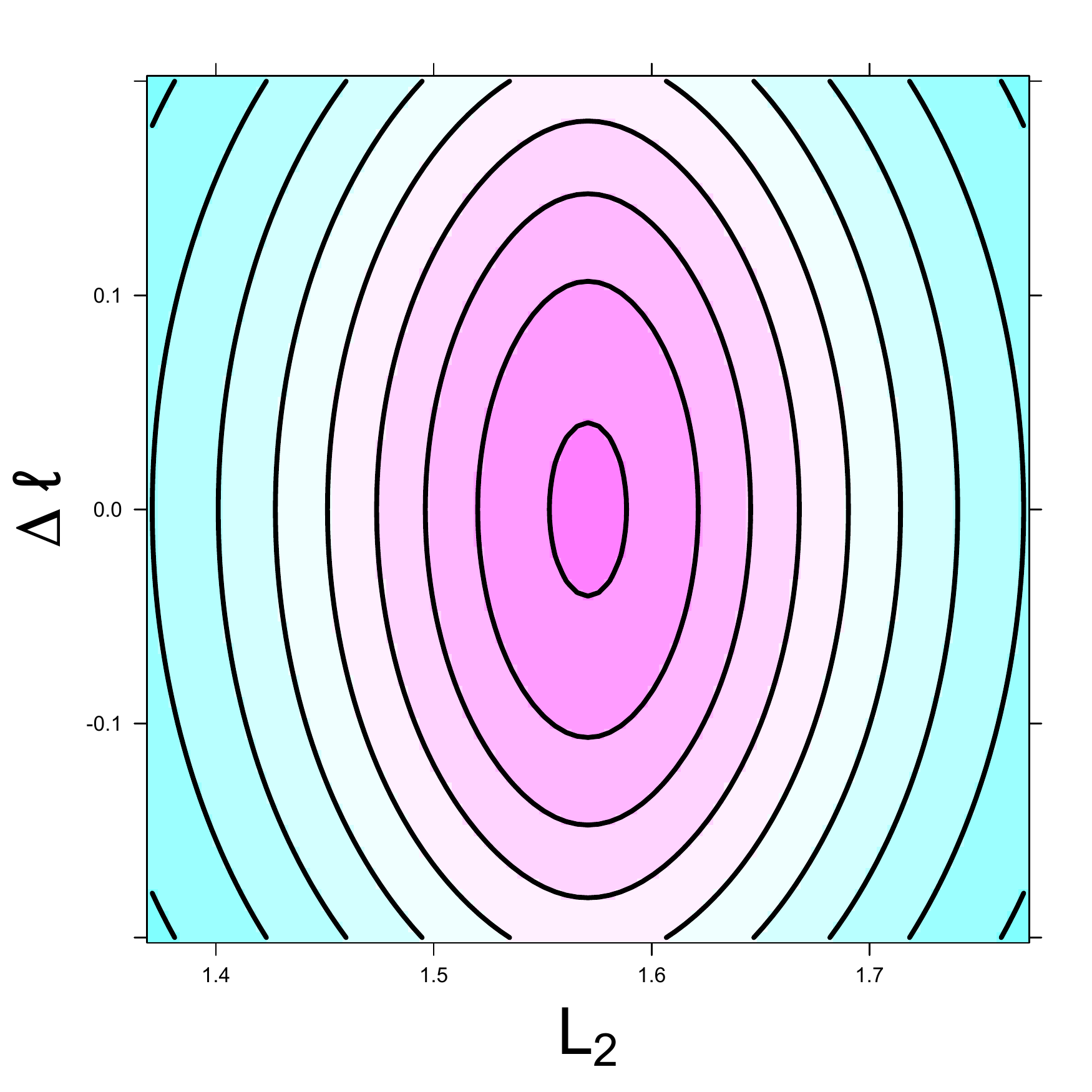}       \includegraphics[scale=0.25]{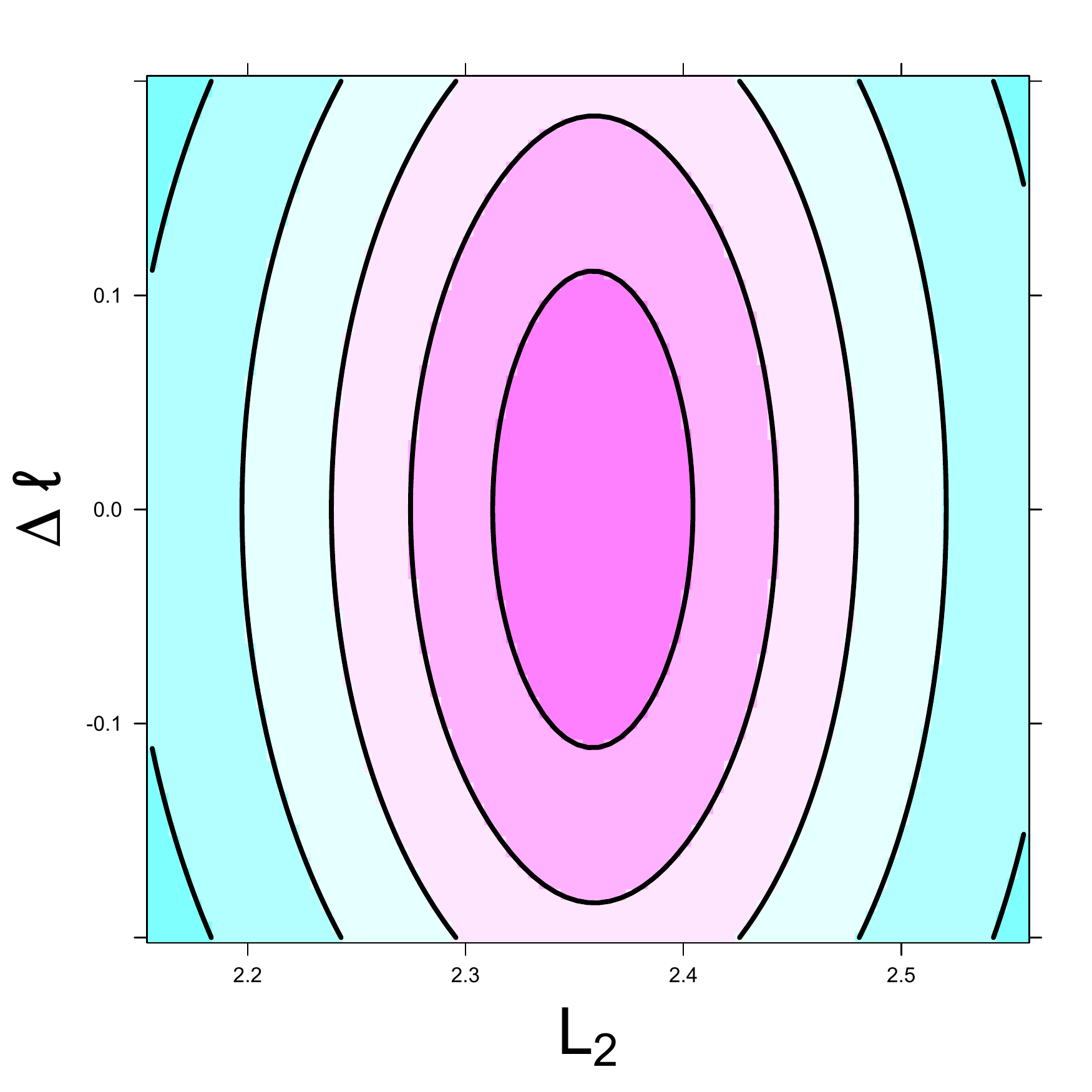}   \\

\includegraphics[scale=0.25]{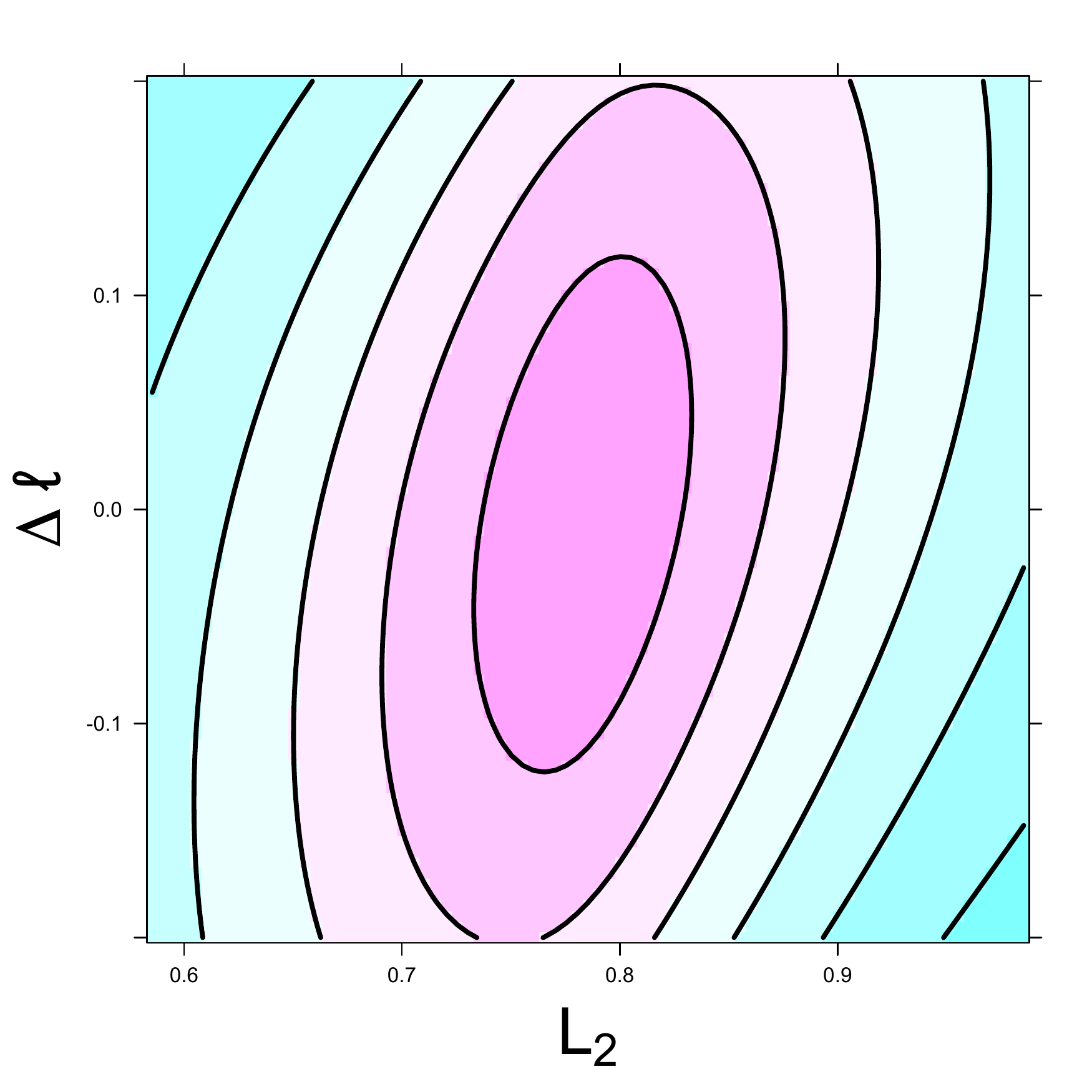}     \includegraphics[scale=0.25]{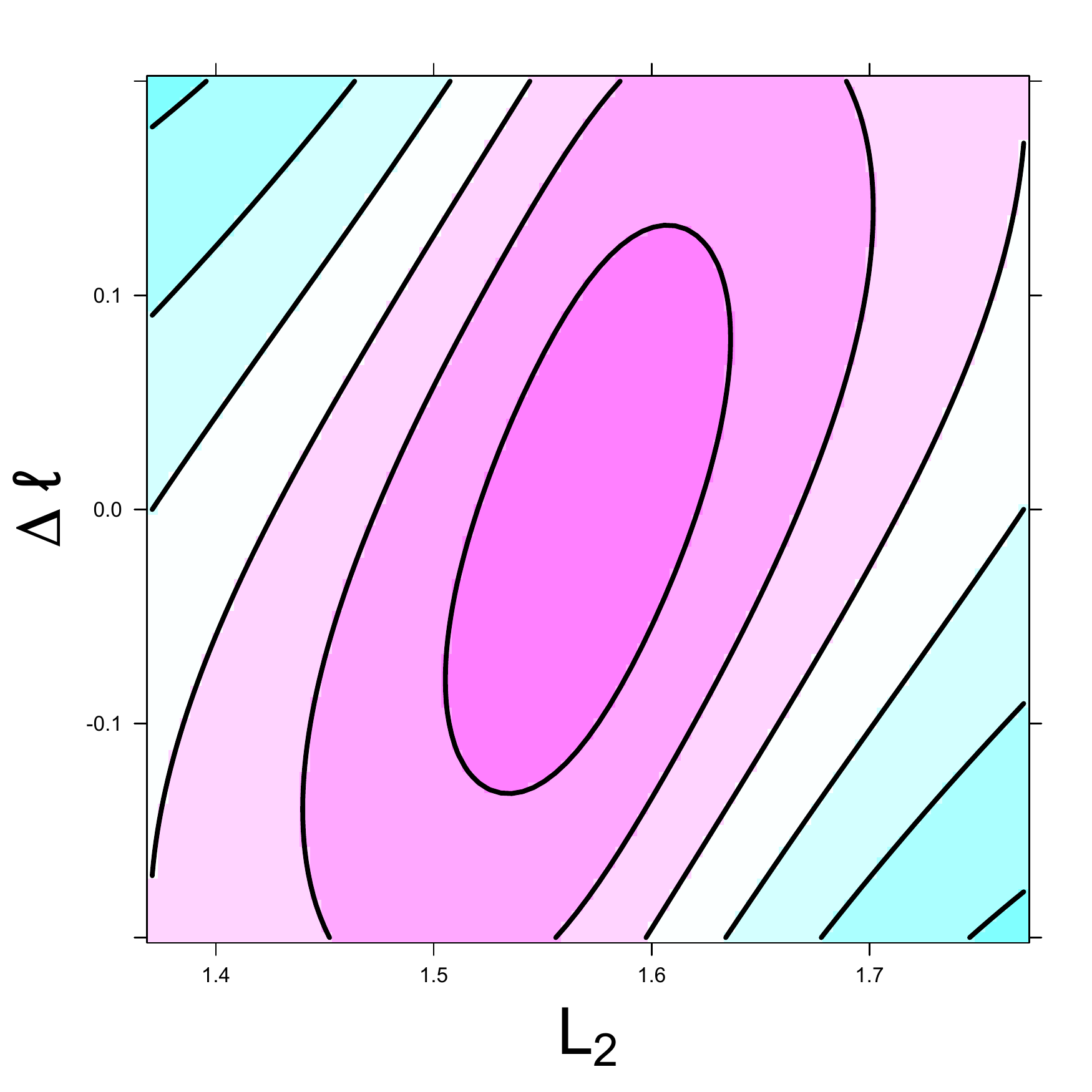}       \includegraphics[scale=0.25]{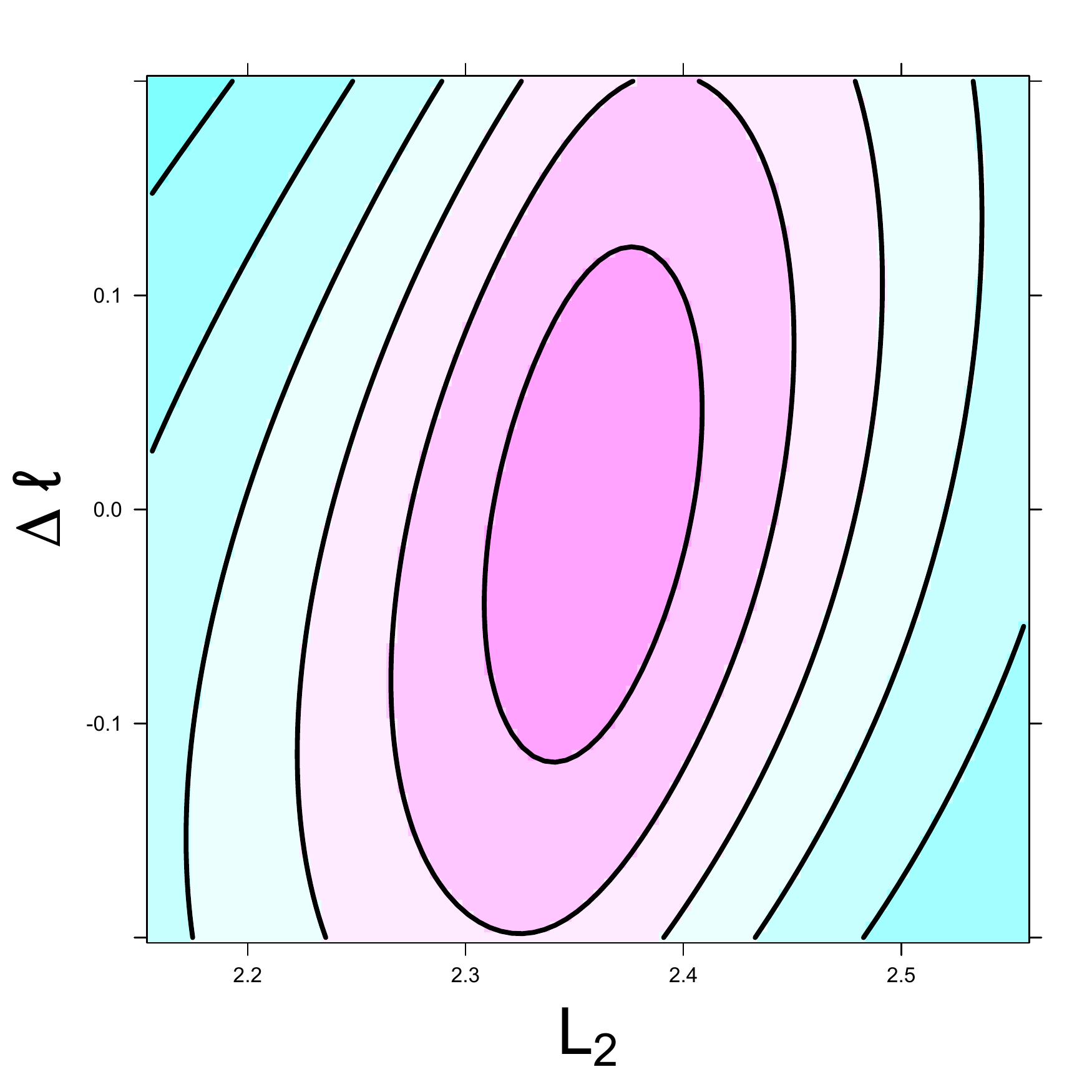}

\includegraphics[scale=0.25]{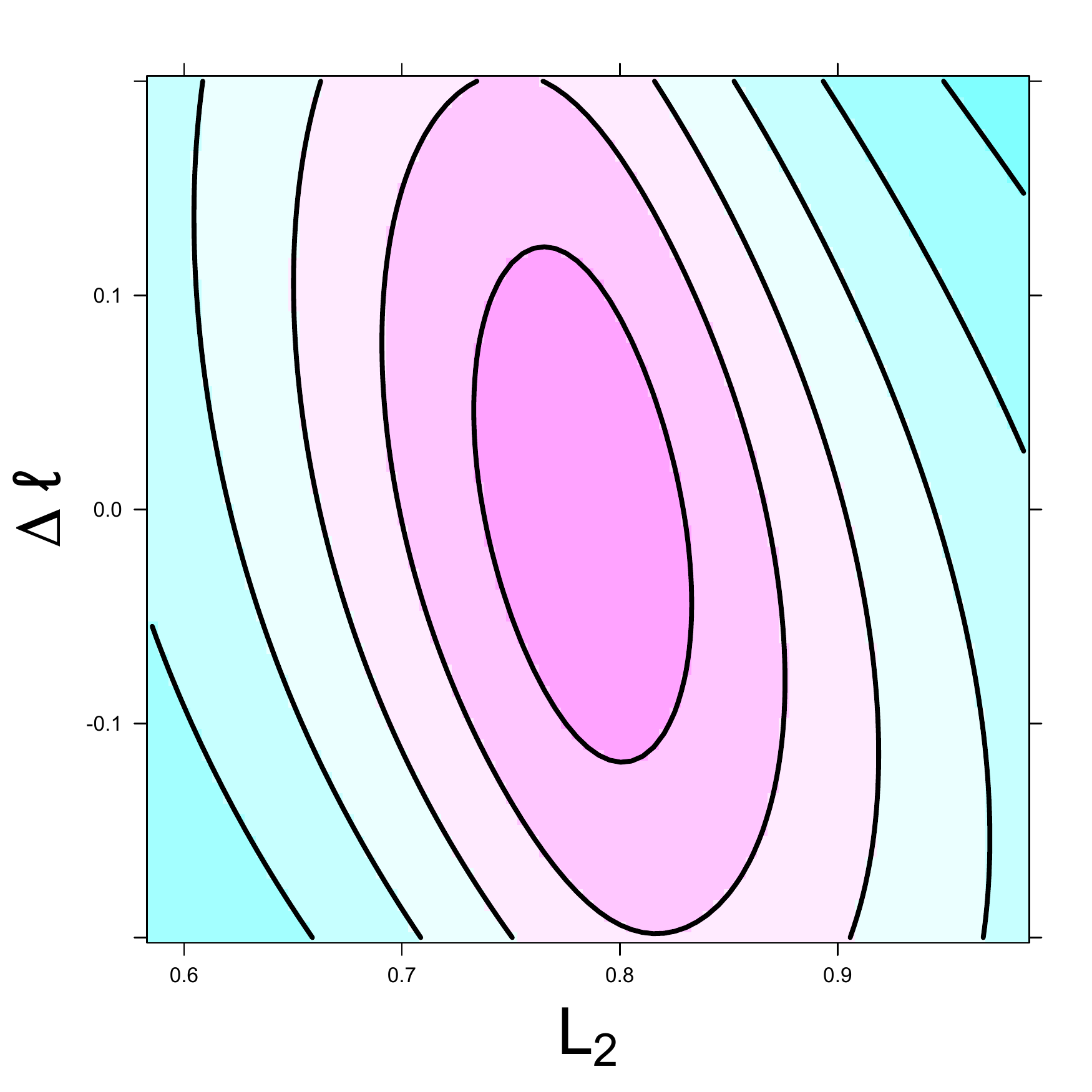}     \includegraphics[scale=0.25]{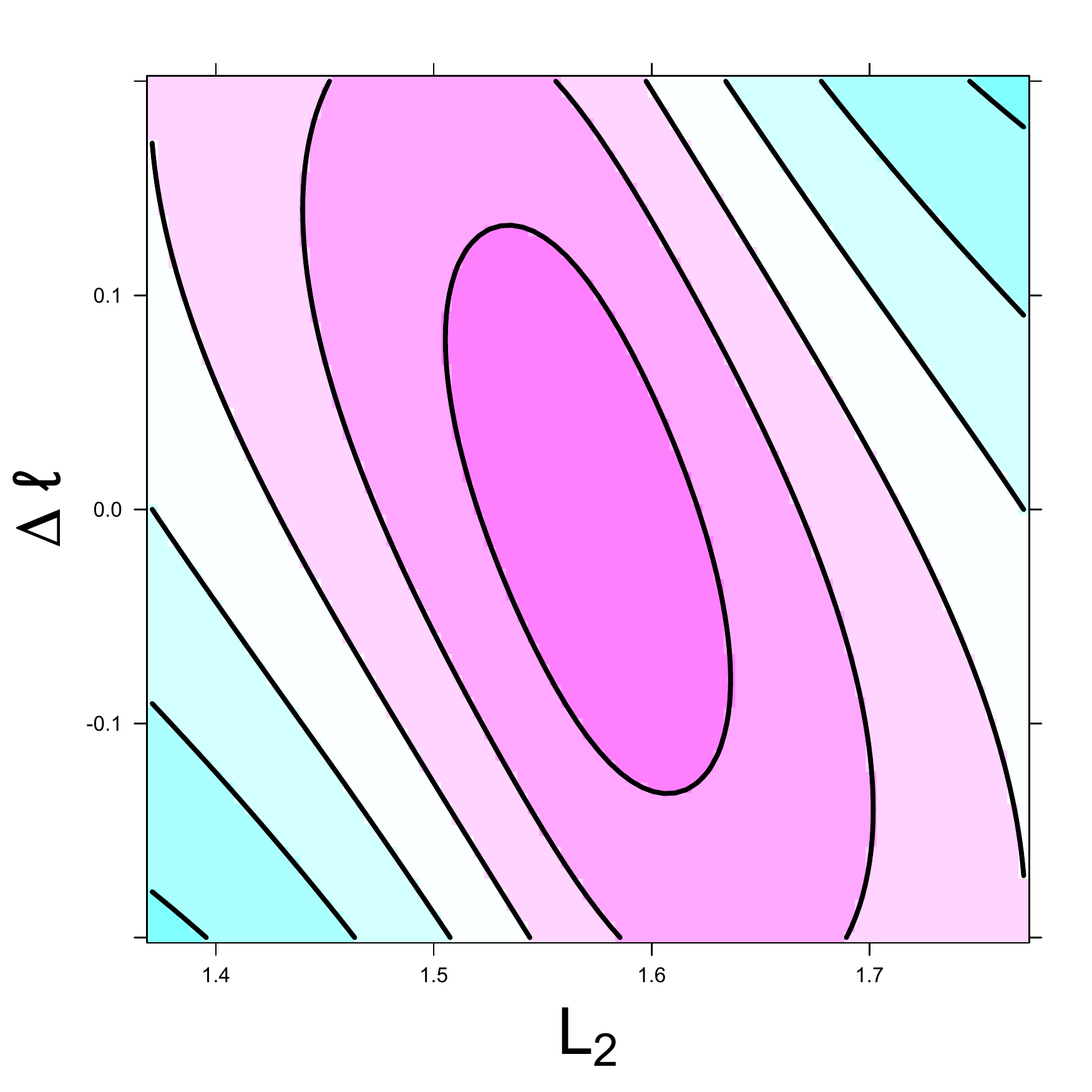}       \includegraphics[scale=0.25]{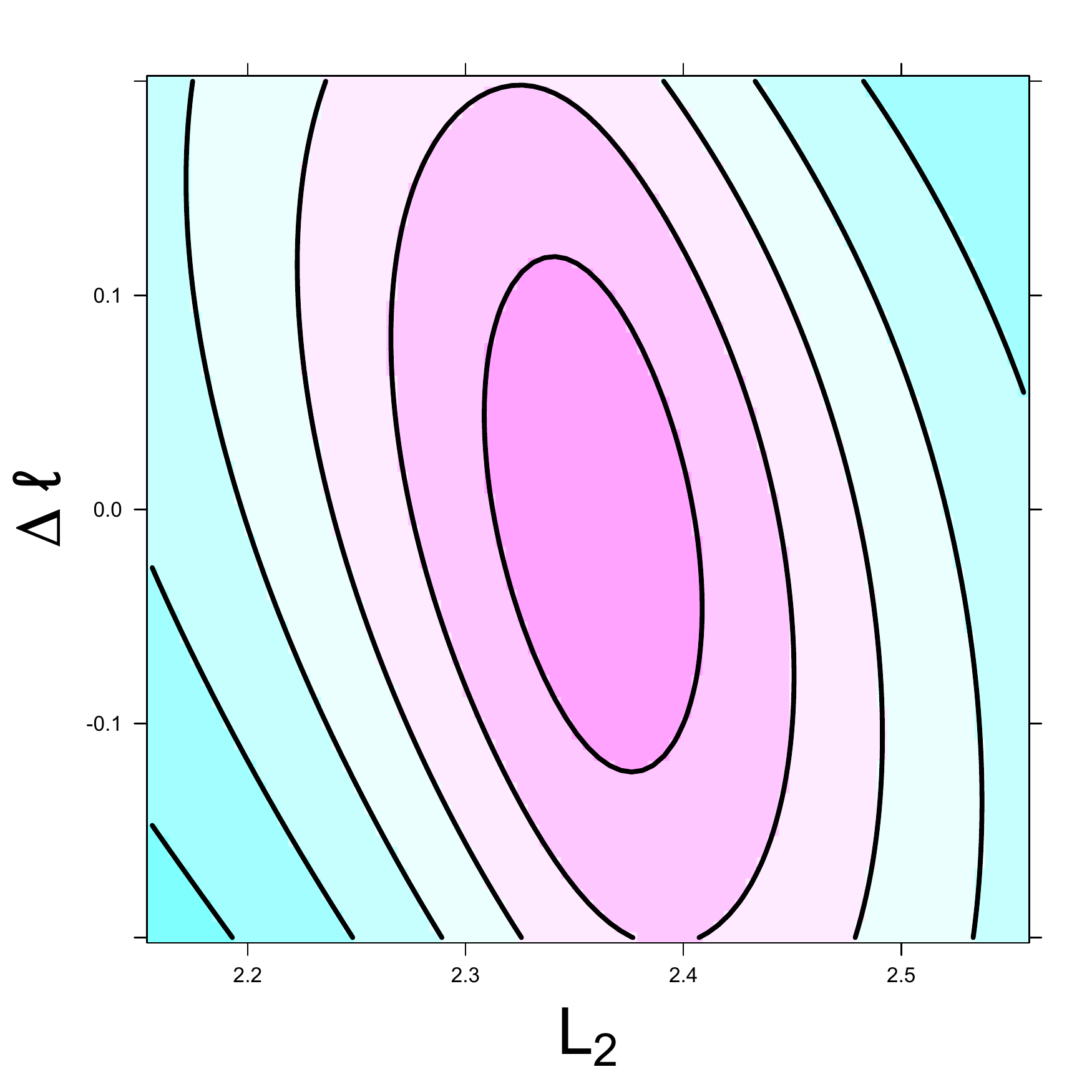}

\caption{Contour plots of the covariance function given in Example 1 after adding the antisymmetric part (\ref{parte-antisimetrica}) in terms of $L_2$ and $\Delta \ell$ for fixed values of $L_1 = \pi/3, \pi/2, 2\pi/3$ (from left to right). Here, $\tau^2=100$, $\nu=1.5$ and $\alpha = 8$. The first row corresponds to a longitudinally reversible process ($\kappa=0$), whereas the second and third rows correspond to longitudinally irreversible processes with $\kappa = 1$ and $\kappa=-1$, respectively. Light blue colors indicate lower values.}
        \label{contorno2}
\end{figure}

Figure \ref{contorno3} shows the contour plots of this covariance function in terms of $L_1$ and $L_2$ for different fixed values of $\Delta\ell = -0.2,0,0.2$, with $\tau^2=100$, $\nu=1.5$, $\alpha = 8$, and $\kappa=1$. Each panel considers $(L_1,L_2)\in [1.2,1.9]^2$. It is clear that the antisymmetric part produces a shift in the contour plots. The covariance function differs for $\Delta\ell=-0.2$ and $\Delta\ell=0.2$, illustrating the longitudinal irreversibility generated by our proposal. This asymmetry could also be observed in the second and third rows of Figure \ref{contorno2} because the graphs do not reflect across the horizontal line given by $\Delta\ell = 0$.

\begin{figure}
  \centering
\includegraphics[scale=0.25]{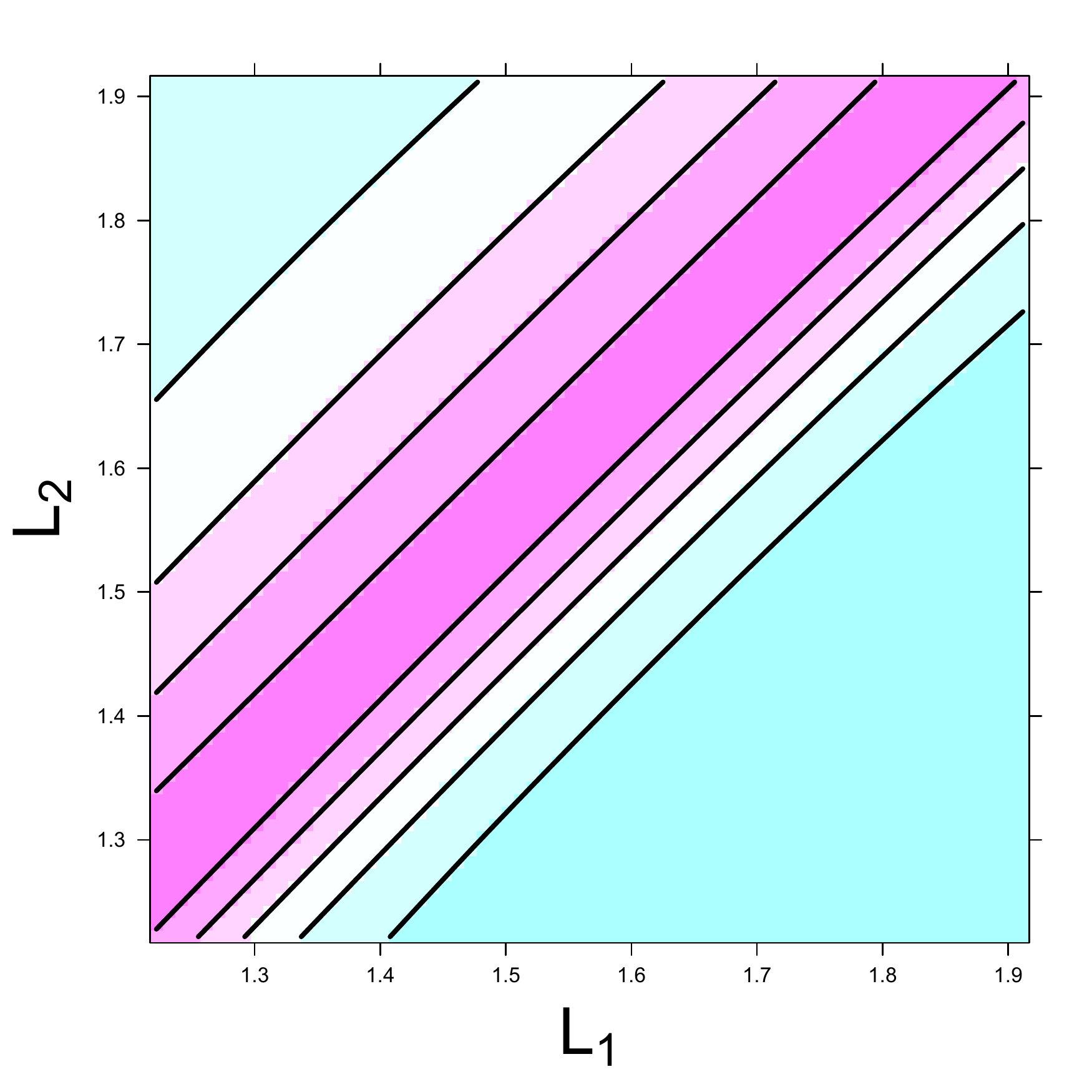}   \includegraphics[scale=0.25]{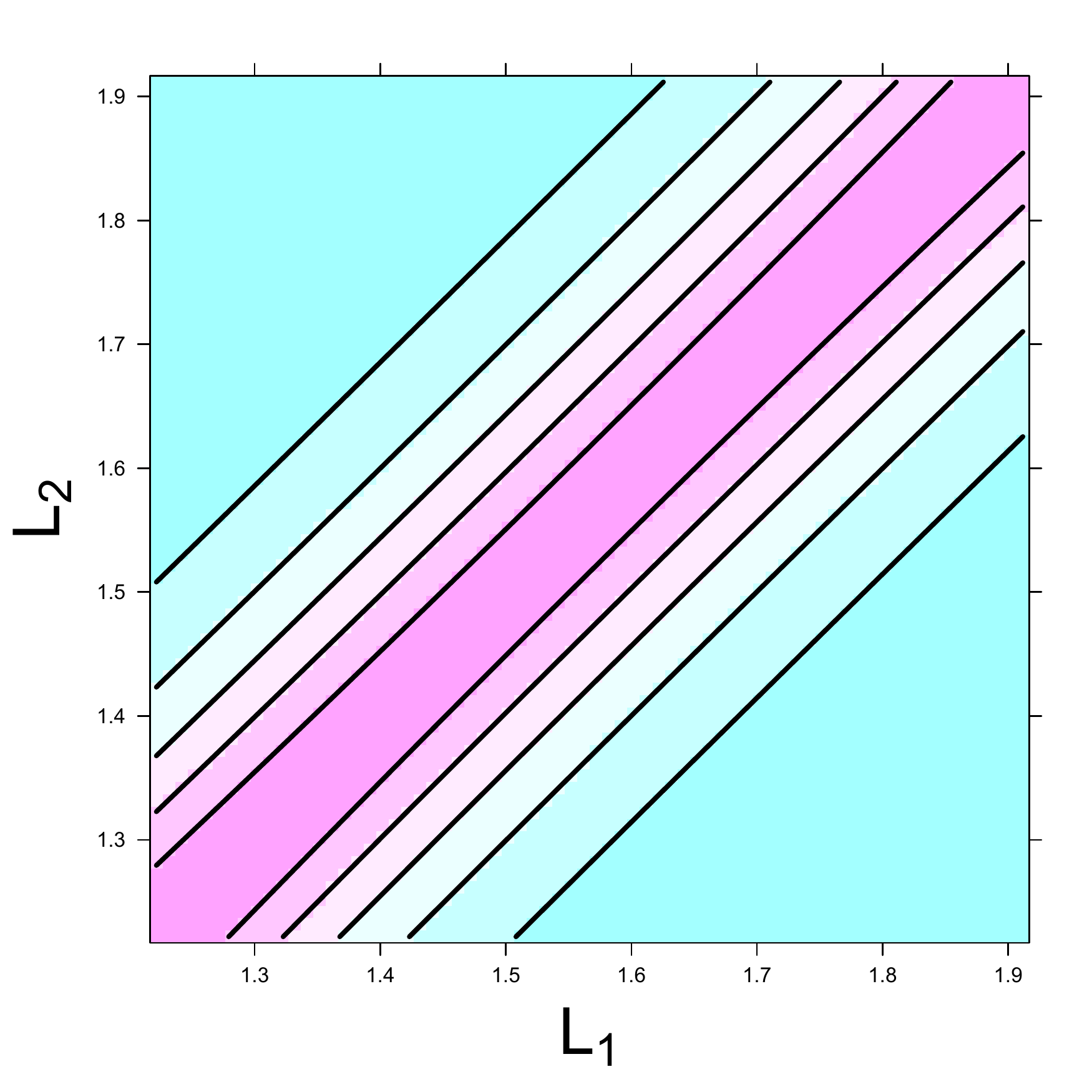}     \includegraphics[scale=0.25]{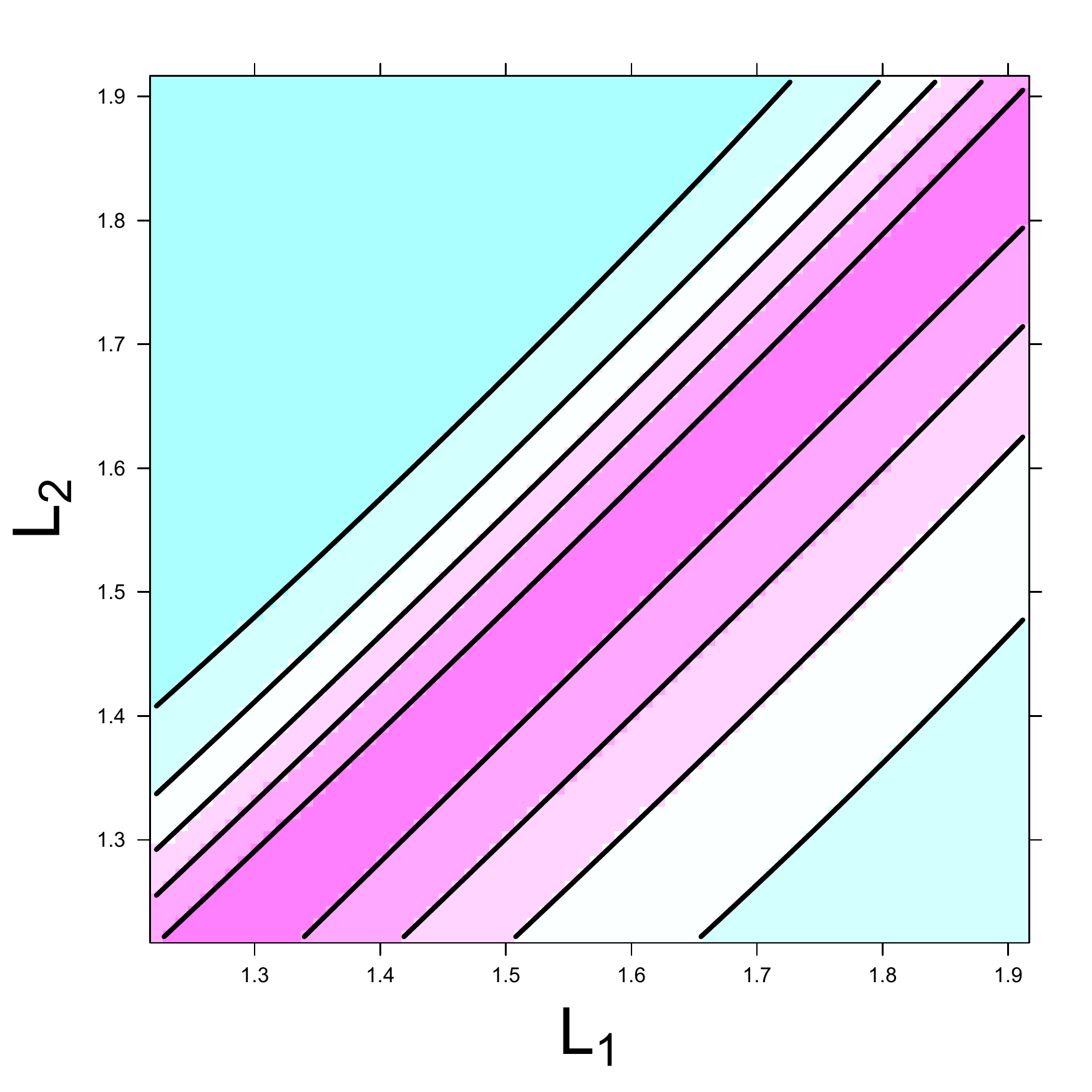}

\caption{Contour plots of the covariance function given in Example 1 after adding the antisymmetric part (\ref{parte-antisimetrica}) in terms of $L_1$ and $L_2$, for $\Delta \ell = -0.2$ (left), $\Delta \ell = 0$ (middle) and $\Delta \ell = 0.2$ (right). Here, $\tau^2=100$, $\nu=1.5$, $\alpha = 8$, and $\kappa=1$. Light blue colors indicate lower values.}
        \label{contorno3}
\end{figure}

In Figure \ref{fig2}, we report the realizations from the covariance function given in Example 1 after adding $g_m(n,n')$ as in (\ref{parte-antisimetrica}). Specifically, we consider $(\alpha, \kappa) \in \{2, 4\}\times \{0, 1\}$. The inclusion of the antisymmetric part generates processes with a stronger anisotropy in the northwest direction than in other directions. This behavior is consistent with the contour plots previously reported.

\begin{figure}
  \centering
\includegraphics[scale=0.1]{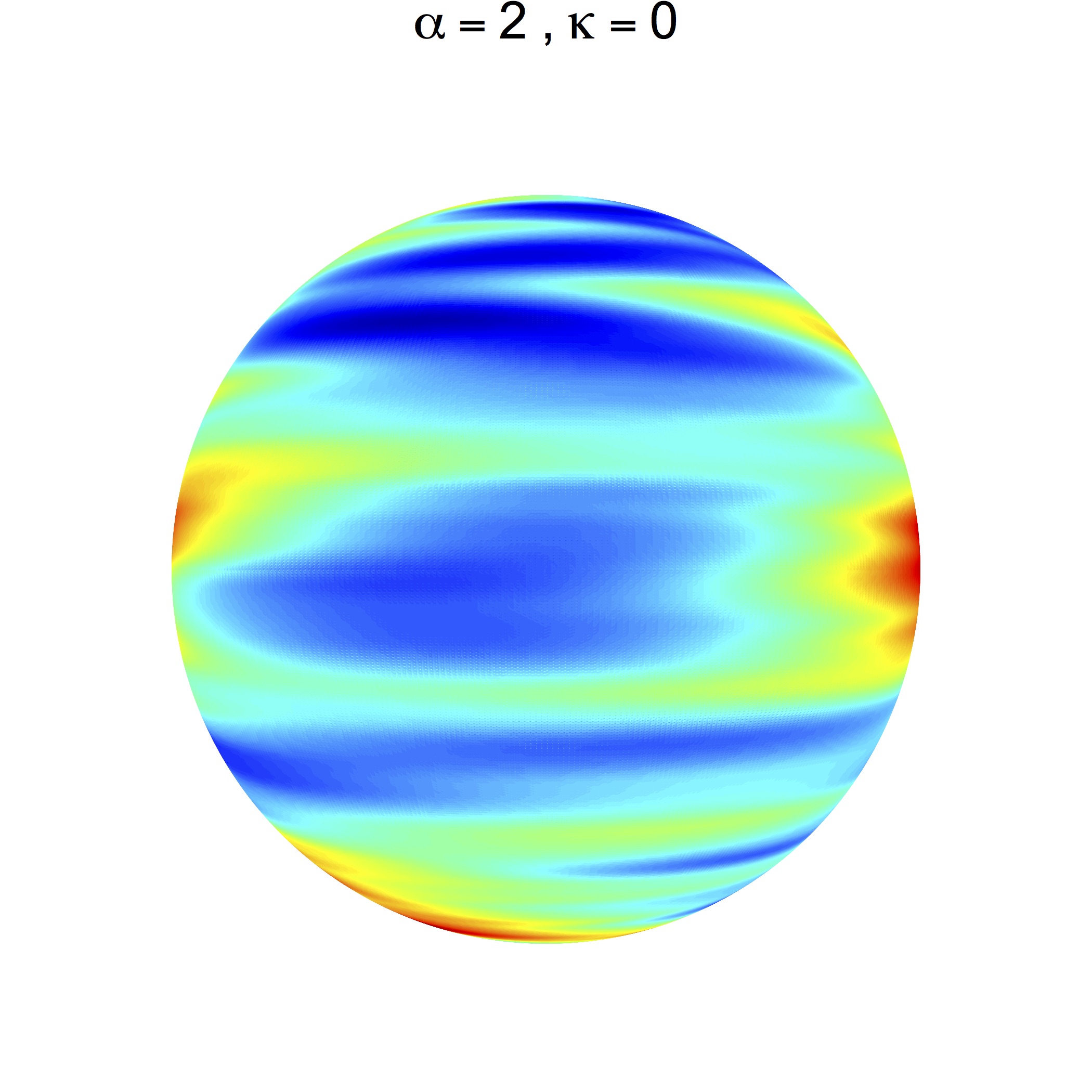}    \includegraphics[scale=0.1]{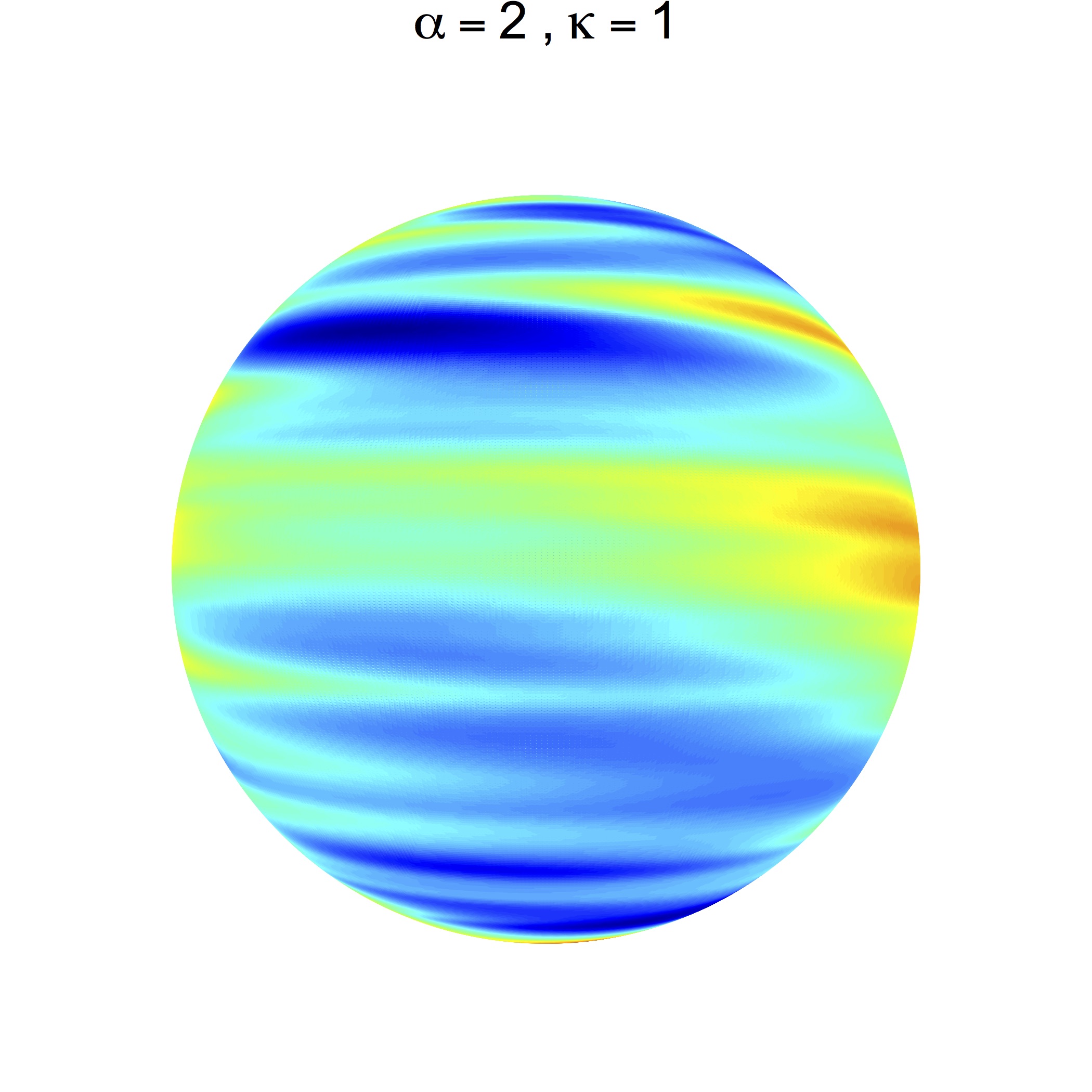}
\includegraphics[scale=0.1]{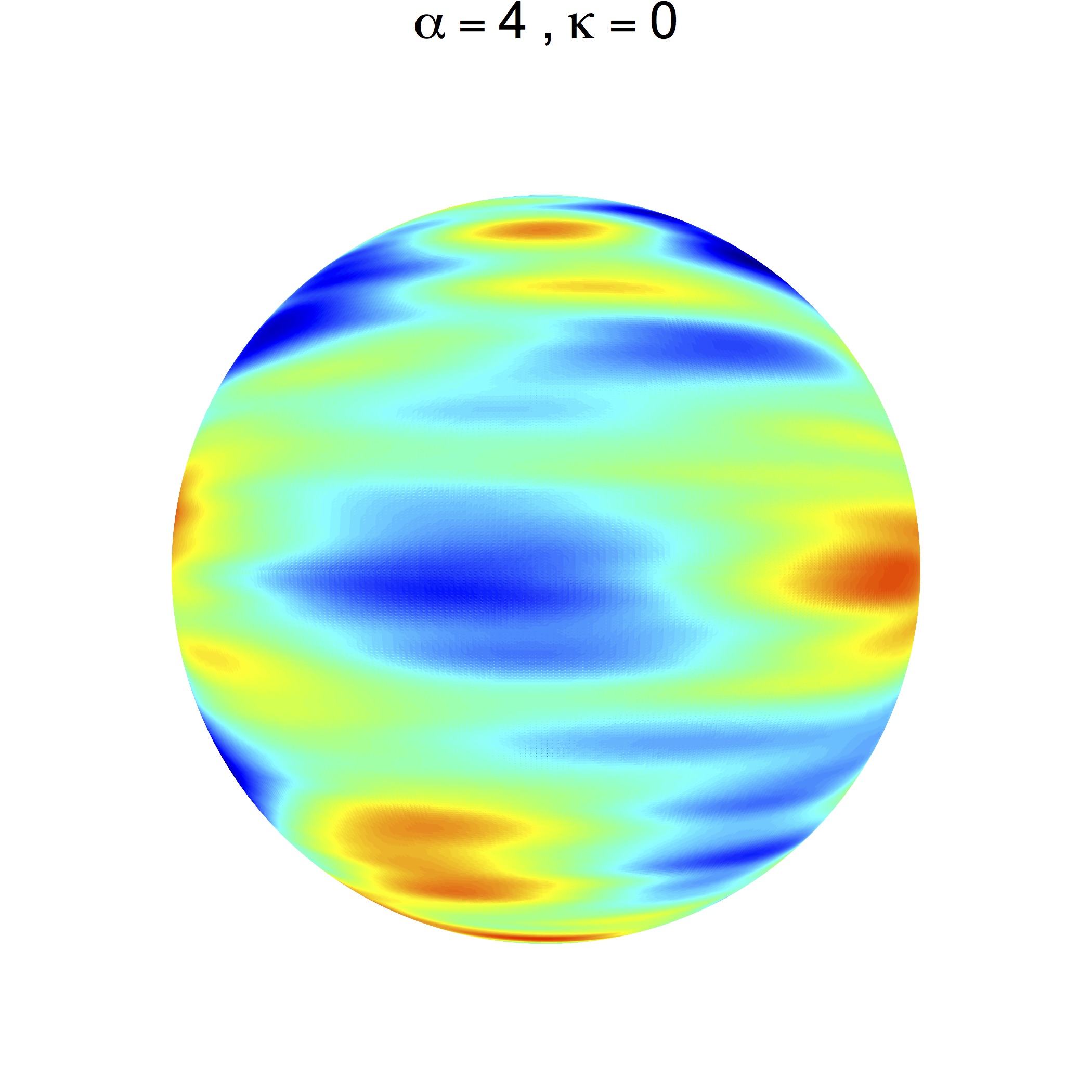}    \includegraphics[scale=0.1]{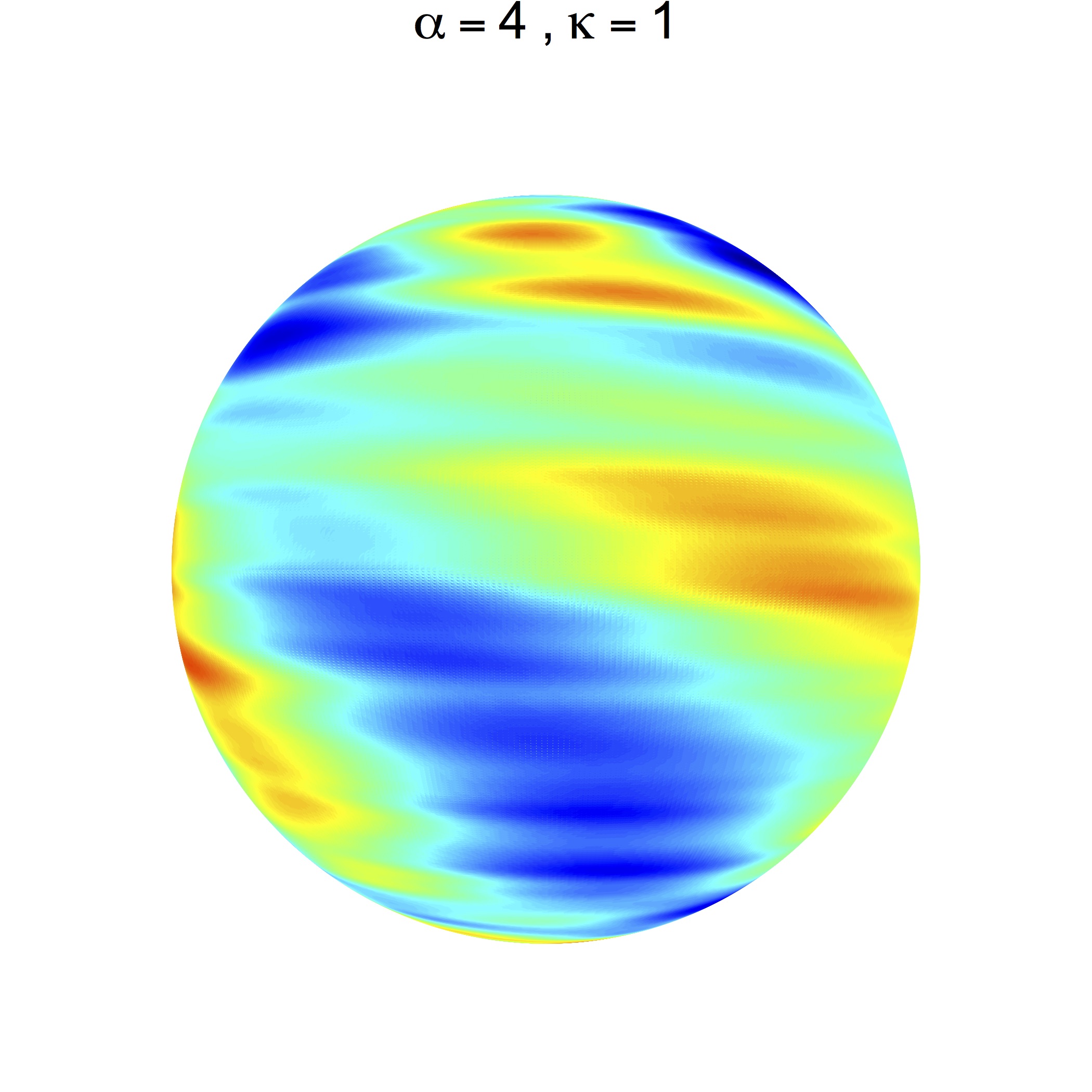}
\caption{Simulated processes on a grid of longitudes and latitudes of size $500\times 500$, with $N=200$, and the covariance function in Example 1 after adding the antisymmetric part (\ref{parte-antisimetrica}). We set $\nu = 1.5$, $\tau^2 = 100$, and different values for $\alpha$ and $\kappa$. The same random seed has been used for each realization.}
        \label{fig2}
\end{figure}

\subsection{Assessing the Accuracy of the Simulation Algorithm}

We turn to a validation study of the quality of the simulation algorithm. We first study how the algorithm reproduces the theoretical variogram structure. We simulate $1000$ independent realizations over $1000$ spatial locations, considering the covariance function given in Example 1, with $\nu=1.5$, $\tau^2=100$ and $\alpha=10$. Then, the empirical local variograms are obtained for fixed latitudes and compared to the theoretical variograms. Figure \ref{fig3} displays the results for four distinct latitudes. Note that on average, the empirical variograms match the theoretical models. The variability of the empirical variograms increases as the longitudinal lag increases, which is commonly observed in practice (see, e.g., \citealp{cuevas2020fast}). Additionally, for latitudes close to the south pole, the variabilities of the empirical variograms are more severe than those far from the south pole.

We also verify the theoretical bound for the $L^2$-error via simulations. We consider the same setting as in the previous examples, i.e.; in particular we have that $\nu=1.5$, which corresponds to a quadratic algebraic decay of the error. The true process is taken as the Karhunen-Lo\`eve expansion with $N=1000$, since for larger $N$, we do not observe substantial variations. Then, we progressively truncate the expansion at different values of $N$ and look at the decay of the error on a logarithmic scale. Following \cite{lang2015isotropic}, instead of the $L^2$-error in space, we quantify a stronger error given by the maximum error over all grid points. Figure \ref{fig4} shows that, on average, we obtain the expected convergence rate for $1000$ independent repetitions of this experiment. More precisely, the empirical convergence rate is given by $N^{-2.009}$, which is close to the theoretical convergence rate.

\begin{figure}
  \centering
\includegraphics[scale=0.35]{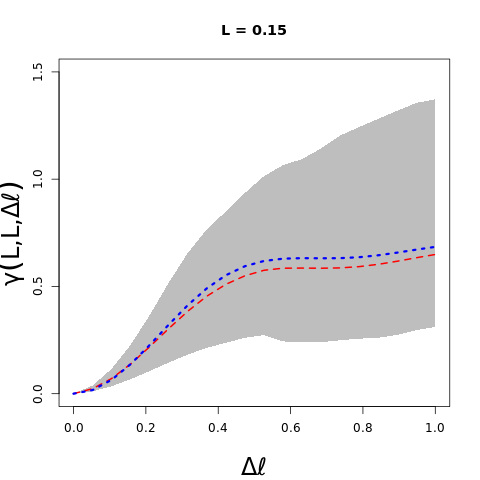}    \includegraphics[scale=0.35]{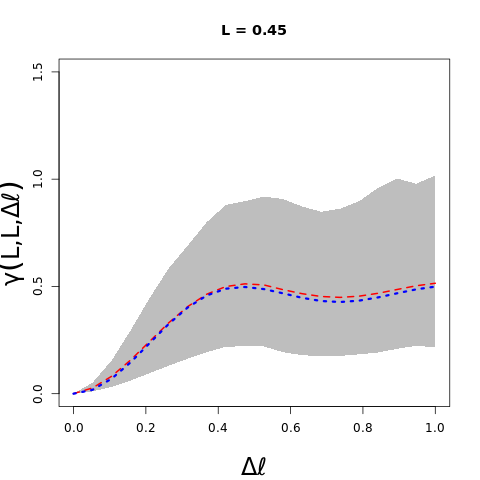}
\includegraphics[scale=0.35]{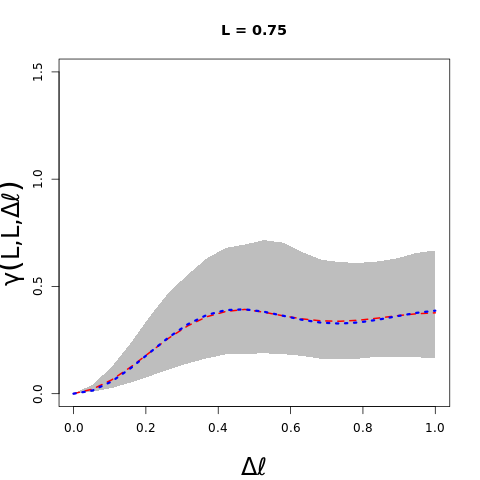}      \includegraphics[scale=0.35]{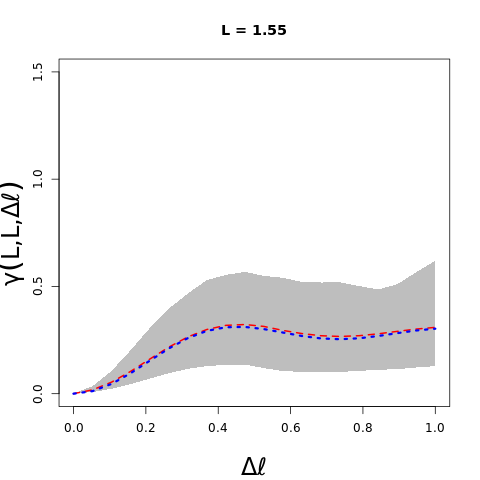}
\caption{Empirical variograms for $1000$ independent simulations from Example 1 versus theoretical variograms. Here, we consider four distinct latitudes and the parameters $\nu=1.5$, $\tau^2=100$ and $\alpha=10$. For each panel, the red dashed line is the theoretical variogram, the blue dotted line is the average empirical variogram, and the gray zones show the empirical variogram envelopes based on the $1000$ repetitions. }
        \label{fig3}
\end{figure}

\begin{figure}
\centering
\includegraphics[scale=0.36]{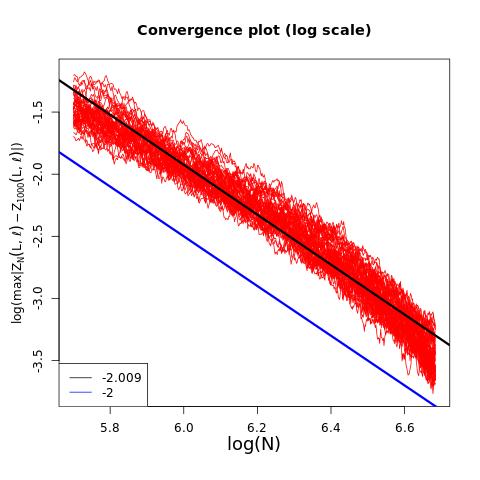}
\caption{Convergence rates in terms of $N$, in a logarithmic scale, for $1000$ independent realizations. The theoretical rate is given by the blue line. The average empirical rate, associated with the $1000$ repetitions, is given by the black line.}
\label{fig4}
\end{figure}

\section{Discussion}
\label{discussion}

In this paper, we discuss several aspects related to Karhunen-Lo\`eve expansions of axially symmetric Gaussian processes. We illustrate how to obtain the limit cases, isotropy and longitudinal independence, by means of an adequate choice of the Karhunen-Lo\`eve coefficients. We have also incorporated an antisymmetric coefficient, which allows for the parametric regulation of longitudinal reversibility. Bounds for the $L^2$-error associated with a truncated version of the Karhunen-Lo\`eve expansion have been derived. This weighted sum of finitely many spherical harmonic functions serves as a natural simulation strategy. Simulation experiments are performed that show the effectiveness of our proposal: (1) it reproduces the prescribed second-order dependency, and (2) the empirical convergence rate of the truncation error matches the theoretical one.

The investigation of more complex coefficients and their impact on the attributes of the axially symmetric process is an interesting topic that merits more attention. The smoothness and H\"older continuity properties could be explored in a similar fashion to the works of \cite{lang2015isotropic}, \cite{kerkyacharian2018regularity} and \cite{cleanthous2020regularity}. The simultaneous modeling of multiple correlated spatial processes on spheres, each one having an axially symmetric structure, is also a promising research direction. The findings of \cite{jun2011non}, \cite{alegria2019covariance} and \cite{Emery2019} might be useful here. Exploring axially symmetric processes that evolve temporally is another interesting topic.

Our findings are not limited to the simulation of axially symmetric processes and may certainly be used for both the modeling and prediction of global data. The search for accurate and efficient methods to estimate the parameters involved in our models is a challenging topic that we expect to tackle in the future.

\section*{Acknowledgments}

The research work of Alfredo Alegr\'ia was partially supported by Grant FONDECYT 11190686 from the Chilean government.

\bibliographystyle{apalike}
\bibliography{mybib}

\end{document}